\documentclass{amsart}

\usepackage[latin1]{inputenc}
\usepackage{amsfonts,amssymb,amsmath,enumerate}
\usepackage{verbatim}
\usepackage{color}
\input{xypic}
\theoremstyle{plain}
\newtheorem{theorem}{Theorem}[section]

\newtheorem{lemma}[theorem]{Lemma}
\newtheorem{proposition}[theorem]{Proposition}
\newtheorem{conjecture}[theorem]{Conjecture}

\theoremstyle{definition}
\newtheorem{definition}[theorem]{Definition}

\newtheorem{examplewr}[theorem]{Example}
\newtheorem{ass}[theorem]{Assumption}
\theoremstyle{remark}
\newtheorem{obswr}[theorem]{Observation}
\newtheorem{remarkwr}[theorem]{Remark}

\newenvironment{remark}{\begin{remarkwr}\begin{upshape}}{\end{upshape}\end{remarkwr}}

\renewcommand\leq{\leqslant}\renewcommand\geq{\geqslant}

\newcommand{\sha}{L\!\!L\!\!I}

\newcommand{\zreal}{\mathbb{R}}

\newcommand{\zinteger} {\mathbb{Z}}
\newcommand{\zrational} {\mathbb{Q}}

\newcommand{\zZ}{\zinteger}
\newcommand{\zQ}{\zrational}

\newcommand{\zR}{\zreal}

\newcommand\Z{\ensuremath{\mathbb Z}}\newcommand\A{\ensuremath{\mathbb A}}
\newcommand\Q{\ensuremath{\mathbb Q}}\newcommand\R{\ensuremath{\mathbb R}}
\newcommand\C{\ensuremath{\mathbb C}}\newcommand\F{\ensuremath{\mathbb F}}
\newcommand\Qb{{\overline\Q}}
\newcommand\cM{\ensuremath{\mathcal M}}

\newcommand\disc{\operatorname{disc}}

\newcommand\End{\operatorname{End}}

\newcommand\Frob{\operatorname{Frob}}
\newcommand\Gal{\operatorname{Gal}}
\newcommand\GL{\operatorname{GL}}
\newcommand\Hom{\operatorname{Hom}}

\newcommand\im{\operatorname{im}}
\newcommand\Ind{\operatorname{Ind}}

\newcommand\M{\operatorname{M}}

\newcommand\ord{\operatorname{ord}}

\newcommand\Res{\operatorname{Res}}

\newcommand\Tr{\operatorname{Tr}}

\newcommand\Norm{\operatorname{Nm}}
\def\M{\operatorname{M}}

\def\p{\mathfrak p}\def\P{\mathbb P}

\newcommand{\comp}{\begin{picture}(6,5)(-3,-2)\put(0,1){\circle{2}} \end{picture}}\def\circ{\comp}

\newcommand{\ra}{\rightarrow}

\newcommand{\lra}{\longrightarrow}

\DeclareMathOperator{\sign}{sign}
\DeclareMathOperator{\rank}{rank}
\newcommand{\cO}{\mathcal{O}}

\begin{document}

\title[Almost totally complex points on elliptic curves]{Almost totally complex points on elliptic curves}

\author{Xavier Guitart, Victor Rotger and Yu Zhao}

\begin{abstract} Let $F/F_0$ be a quadratic extension of totally real number fields, and let  $E$ be an elliptic curve
over $F$ which is isogenous to its Galois conjugate over $F_0$. A quadratic extension $M/F$ is said to be almost
totally complex (ATC) if all archimedean places of $F$ but one extend to a complex place of $M$.  The main goal of this note is to
provide a new construction of a supply of Darmon-like points on $E$, which are conjecturally  defined over certain
ring class fields of $M$. These points are constructed by means of an extension of  Darmon's ATR method to higher
dimensional modular abelian varieties, from which they inherit the following features: they are algebraic provided
Darmon's conjectures on ATR points hold true, and  they are explicitly computable, as we illustrate with a detailed
example that provides certain numerical evidence for the validity of our conjectures.
\end{abstract}

\address{X. G.: Departament de Matem\`{a}tica Aplicada II, Universitat Polit\`{e}cnica de Catalunya, C.
Jordi Girona 1-3, 08034 Barcelona, Spain and Max Planck Institute for Mathematics, Vivatsgasse 7, 53111 Bonn, Germany}
\email{xevi.guitart@gmail.com}

\address{V. R.: Departament de Matem\`{a}tica Aplicada II, Universitat Polit\`{e}cnica de Catalunya, C. Jordi Girona
1-3, 08034 Barcelona, Spain}
\email{victor.rotger@upc.edu}

\address{Y.Z.: Department of Mathematics, John Abbott College, Montreal, Quebec, H9X 3L9, Canada } 
\email{yu.zhao@johnabbott.qc.ca}
%\subjclass{11G18, 14G35}

\maketitle

\tableofcontents

\section{Introduction \label{zintroduction}}

Let $E$ be an elliptic curve defined over a number field $F$ and, for any field extension $K/F$, let $L(E/K, s)$ denote
the Hasse-Weil $L$-function of the base change of $E$ to $K$, which is known to converge on the half-plane $\{ s\in \C:
\mathrm{Re}(s)>\frac{3}{2}\}$.

The Mordell-Weil theorem asserts that the abelian group $E(K)$ of $K$-rational points on $E$ is finitely generated, that
is to say,
$$
E(K) \simeq T \times \zZ^r,
$$
where $T$ is a finite group and $r=r(E/K)\geq 0$ is a non-negative integer, which is called the {\em Mordell-Weil rank}
of
$E/K$.

There are two conjectures which stand out as cornerstones in the arithmetic of elliptic curves:

\vspace{0.2cm}

{\bf Conjecture (MOD).} {\em The elliptic curve $E/K$ is modular: there exists an automorphic representation $\pi$ of
$\GL_2(\A_K)$ such that $L(E/K,s-\frac{1}{2})=L(\pi,s)$. In particular, $L(E/K,s)$ can be analytically continued to an
entire function on the complex plane and it satisfies a functional equation relating the values at $s$ and $2-s$.}

\vspace{0.2cm}

{\bf Conjecture (BSD).} Assume that (MOD) holds for $E/K$ and let $r_{an}(E/K)=\ord_{s=1} L(E/K,s)$ denote the order of
vanishing  of $L(E/K, s)$ at $s=1$, which we call the {\em analytic rank} of $E/K$. Then
\begin{equation*}
r(E/K)\overset{?}{=}r_{an}(E/K).
\end{equation*}

Conjecture (MOD) is nowadays known to hold, under mild hypothesis, when $F$ is totally real and $K/F$ is Galois with
solvable Galois group, thanks to the work of Wiles, Skinner-Wiles, Langlands and others. More precisely, when $F$ is
totally real, $E$ is known to be modular by \cite{wiles}, \cite{bcdt}, \cite{SkWi}, unconditionally if the base field is
$F=\Q$ and under some technical conditions on the reduction type at the primes of $F$ above $3$ when $[F:\Q]>1$. In this
setting, this amounts to saying that there exists a Hilbert modular eigenform $f_E$ of parallel weight $2$ over $F$ such
that $L(E/F,s)$ is equal to the L-function $L(f_E,s)$ associated with that form. If $K/F$ is solvable, then (MOD)
follows from the modularity of $E$ over $F$ by applying Langlands's cyclic base change. If $F=\Q$ and $K$ is a totally
real Galois number field, recent work of Dieulefait \cite{Di} proves (MOD) under simple local assumptions on $K$, and
one can expect that similar techniques may lead in the future to a similar result for arbitrary totally real fields $F$.

In light of these results, we assume throughout that $F$ is totally real and $E$ is modular. Let $\mathfrak N$ denote
the conductor of $E$, an integral ideal of $F$, which for simplicity we assume to be square-free.

Thanks to the work of Kolyvagin, Gross-Zagier and Zhang, Conjecture (BSD) is then known to hold when $K$ is either $F$
or a totally imaginary extension of $F$, $(\mathfrak N,\disc(K/F))=(1)$, $r_{an}(E/K)\leq 1$ and the {\em
Jacquet-Langlands (JL) hypothesis} holds:

\begin{enumerate}
\item[(JL)] Either $[F:\Q]$ is odd or $\mathfrak N \ne (1)$.
\end{enumerate}

In particular, when $K$ is a totally imaginary extension of $F$ and $r_{an}(E/K)=1$, the above result implies that if
(JL) is satisfied, there exists a non-torsion point in $E(K)$. Precisely when (JL) holds, such a point $P_K$, a
so-called {\em Heegner point}, can be manufactured by means of the theory of complex multiplication on Shimura curves,
and it is Gross-Zagier \cite{gross-zagier} and Zhang \cite{Zh} who showed that the hypothesis $r_{an}(E/K)=1$ implies
that $P_K$ is not torsion. Finally, Koyvagin's method \cite{kolyvagin} of Euler systems is the device which permits to
show that in fact there are no points in $\Q\otimes E(K)$ which are linearly independent of $P_K$, thereby showing
(BSD). This is made possible thanks to the existence, along with the point $P_K$, of a system
$$
\{ P_c \in E(H_c), c\geq 1, (c,\disc(K/F))=1\}
$$
of rational points on $E$ over the {\em ring class field} $H_c/K$, the abelian extension of $K$ associated by class
field theory to the Picard group $\mathrm{Pic}(\mathcal O_c)$ of invertible ideals in the order $\mathcal O_c\subset K$
of conductor $c$ of $K$.

That this supply of points should exist can be predicted using Conjecture (BSD), even if $K$ is not totally imaginary,
as we now explain. Let $K/F$ be {\em any} quadratic field extension such that $(\mathfrak N,\disc(K/F))=1$. Write
\begin{equation}
\label{eqn:heegner-fact}
\mathfrak N=\mathfrak N^+\cdot \mathfrak N^-,
\end{equation}
where $\mathfrak N^+$ (resp. $\mathfrak N^-$) is the product of the prime divisors of $\mathfrak N$ which split (resp.
remain inert) in $K$.

Let $\chi: \Gal(K^{ab}/K)\ra \C^{\times}$ be a character of finite order and conductor relatively prime to $\mathfrak
N$. Let $r_1(K/F)$ and $r_2(K/F)$ be the number of archimedean places of $F$ which extend to a couple of real (resp.\,to
a complex) place(s) of $K$, so that $[F:\Q]=r_1(K/F)+r_2(K/F)$. Then the sign of the functional equation of the
L-function $L(E/K,\chi,s)$ of $E/K$ twisted by $\chi$ is
\begin{equation}
\label{sign-over-H}
\sign(E/K)= \sign(E/K, \chi) = (-1)^{r_2(K/F)+\sharp \{\wp \mid \mathfrak N^-\}},
\end{equation}
independently of the choice of $\chi$.

For any abelian extension $H/K$, let $\hat{\Gal}(H/K)=\Hom(\Gal(H/K),\C^\times)$ denote the group of characters of
$\Gal(H/K)$. The L-function of the base change of $E$ to $H$ factors as
$$
L(E/H,s) = \prod_{\chi\in \hat{\Gal}(H/K)} L(E/K,\chi,s).
$$

The Birch and Swinnerton-Dyer conjecture (BSD) in combination with \eqref{sign-over-H} gives rise to the following:

\begin{conjecture}\label{BSDheuristics}
Assume $\sign(E/K)=-1$ and let $H/K$ be an abelian extension, unramified at the primes dividing $\mathfrak N$. Then
\begin{equation}
\rank E(H)\overset{?}{=} [H:K],
\end{equation}
if and only if $L'(E/K,\chi,1)\ne 0$ for all $\chi\in \Hom(\Gal(H/K),\C^\times)$.
\end{conjecture}

No proven result is known about Conjecture \ref{BSDheuristics} beyond the achievements of Gross-Zagier, Kolyvagin and
Zhang in the case $r_2(K/F)=[F:\Q]$ mentioned above. In spite of this, a plethora of {\em conjectural} constructions of
points have been proposed so far in various settings beyond the classical one. These points are commonly called {\em
Stark-Heegner points}, or {\em Darmon points}, as it was H. Darmon in \cite{Dar00} who first introduced them.

Since then, several authors \cite{Das}, \cite{Gr}, \cite{LRV}, \cite{DL}, \cite{Ga} have proposed variations of Darmon's
theme, always giving rise to a recipe that allows to attach, to a given abelian extension $H/K$ satisfying the
hypothesis of Conjecture \ref{BSDheuristics}, a point
\begin{equation}\label{PH}
P_H \in E(H_v),
\end{equation}
rational over the completion $H_v$ of $H$ at some finite or archimedean place $v$ of $H$, which is conjectured to
satisfy the following properties:

\begin{enumerate}
\item[(\mbox{SH}1)] $P_H\overset{?}{\in } E(H)$,

\item[(\mbox{SH}2)] For any character  $\chi: \Gal(H/K)\ra \C^\times$, the point $$P_\chi:= \sum_{\sigma \in \Gal(H/K)}
\chi(\sigma)^{-1}\sigma( P_H)\in E(H)\otimes_{\Z} \C$$ is non-zero if and only if $L'(E/K,\chi,1)\ne 0$, and

\item[(\mbox{SH}3)]  there is a reciprocity law describing the action of $\Gal(H/K)$ on $P_H$ in terms of
ideal theory.
\end{enumerate}

The main result of this paper is a new, computable construction of a supply of Darmon-like points in a setting that was
not computationally accessible before. Before describing our contribution in more detail, and being  the constructions
of Darmon points dispersed in the literature, we take the chance to report on the state of
the art of this question. Namely, explain which cases of Conjecture \ref{BSDheuristics} are already covered by the union
of those constructions, and which ones remain intractable.

Keep the above notations and the assumptions of conjecture
\ref{BSDheuristics}, and assume that
$H$ is the narrow ring class field associated with some order in $K$. Then:

\begin{enumerate}

\item[a)] If $r_1(K/F)=0$, $r_2(K/F)=[F:\Q]$, then assumption $\sign(E/K)=-1$ implies that (JL) holds, and conjecture
\ref{BSDheuristics} holds thanks to \cite{gross-zagier}, \cite{kolyvagin} and \cite{Zh}.

\item[b)] If $\sharp \{\wp \mid \mathfrak N^-\}\geq 1$, points $P_H\in E(H_\wp)$ have been constructed in \cite{Dar00},
\cite{Gr} and \cite{LRV}, for which conditions (\mbox{SH}1), (\mbox{SH}2) and (\mbox{SH}3) above have been conjectured.

Some theoretical evidence has been provided for them when $F=\Q$ in \cite{BeDa}, \cite{GSS} and \cite{LoVi}.

Numerical evidence has been given in \cite{Dar00} when $F=\Q$ and $N^-=1$.
%It is object of the Ph.D thesis \cite{Her} to provide numerical evidence for these conjectures when $F=\Q$ and $N^->1$, building on the construction given in \cite{LRV}.

\item[c)] If $r_1(K/F)\geq 1$, $r_2(K/F)\geq 1$ let us distinguish two possibilities:

\begin{enumerate}

\item[c1)] If $r_2(K/F)=1$, $K/F$ is called an {\em almost totally real} (ATR) quadratic extension and we let $v$ denote
the unique archimedean place of $F$ which extends to a complex place of $K$. Then $H_v=\C$ for any place of $H$ above it
and points $P_H\in E(H_v)$ have been constructed in \cite[Ch.\,VIII]{Drpmec}, for which conditions (\mbox{SH}1),
(\mbox{SH}2) and (\mbox{SH}3) above have been conjectured. These conjectures have been tested numerically in \cite{DL}.

\item[c2)] J. Gartner has extended the idea of Darmon \cite[Ch.\,VIII]{Drpmec} to any $K/F$ with $1\leq
r_2(K/F)<[F:\Q]$: in this more general setting, he constructs points $P_H\in E(H_v)$ and again conjectures that
(\mbox{SH}1), (\mbox{SH}2) and (\mbox{SH}3) hold true. His method does not appear to be amenable to explicit
calculations and as a consequence no numerical evidence has been provided for these conjectures.

\end{enumerate}
\end{enumerate}

Note that a), b), c) cover all cases contemplated in Conjecture \ref{BSDheuristics}. Indeed, the only case not covered
by b) arises when $\sharp \{\wp \mid \mathfrak N^-\}=0$, that is, all primes $\wp\mid \mathfrak N$ split in $K$. But
then assumption $\sign(E/K)=-1$ implies that $r_2(K/F)$ is odd, hence $r_2(K/F)\geq 1$. Then a) and c) cover
respectively the case in which $r_1(K/F)=0$ and $r_1(K/F)>0$.

The main contribution of this article is an explicitly computable, construction of a supply of points $P_M\in E(\C)$ in a
setting which lies within c2), but
which is completely different to the one proposed by Gartner. It only works under the following restrictive hypothesis:
\begin{itemize}
\item $F$ contains a field $F_0$ with $[F:F_0]=2$,

\item $E/F$ is $F$-isogenous to its Galois conjugate over $F_0$, and

\item $M$ is an almost totally complex quadratic extension of $F$, that is to say, $r_2(M/F)=[F:\Q]-1$.
\end{itemize}

While this setting is obviously much less general than the one considered in \cite{Ga}, it enjoys the following
features:
\begin{itemize}
\item Numerical approximations to the points $P_M$ are computable, as we illustrate with a fully detailed explicit
example in \S \ref{subsection: numerical examples}.

As we explain in \S\ref{subsection: construction
of the points}, our construction relies on the computation of certain ATR cycles on Hilbert modular varieties. To the best of our knowledge, at present there is available an algorithm for computing such ATR cycles only when the level is {\em trivial} (see \S\ref{subsection: Darmon's algorithm for ATR points} for more details). However, in our setting the level is always nontrivial, and so far this stands as the single issue which prevents our method from being completely automatized.  In the example of \S\ref{subsection: numerical examples} we circumvent the lack of a
general algorithm with an ad hoc computation.
\item We prove that the points $P_M$ belong to $E(M)$ and that they are non-torsion if and only if
$L'(E/M,1)\neq 0$  provided (SH1), (SH2) and (SH3) hold true for  ATR extensions of $F_0$: see Theorem \ref{conjecture:
the point in E is non-torsion} for the precise statement. This is worth remarking, as the conjectures for ATR
extensions can be tested numerically in practice: see \S \ref{subsection: Darmon's algorithm for ATR points} for
a sketch of the algorithm, and \cite{DL}, \cite{GM} for explicit numerical examples.
\end{itemize}

The main source of inspiration for the construction presented here is the previous work \cite{DRZ} of two of the authors with Henri Darmon, in which Heegner points on quotients of the modular curve $X_1(N)$ were used to manufacture ATR points on elliptic curves. 

{\em Acknowledgements.} We are thankful to Jordi Quer for computing for us the equation of the elliptic curve used in \S
\ref{subsec:quer}. Guitart wants to thank the Max Planck Institute for Mathematics for their 
hospitality and financial support during his stay at the Institute, where part of the present work has been carried out.
Guitart and Rotger
received financial support from DGICYT Grant MTM2009-13060-C02-01 and from 2009 SGR 1220.

\section{Quadratic points on modular abelian varieties}\label{section: Quadratic points on modular abelian vareties}

The basis of the main construction of this note --which we explain in \S \ref{secATC}-- lies in Darmon's conjectural theory of points on modular elliptic curves
over almost totally real (ATR) quadratic extensions of a totally real number field.

In a recent article, Darmon's theory has been generalized by Gartner \cite{Gartner-article} by considering {\em quaternionic} modular forms
with respect to not necessarily split quaternion algebras over the base field. Although we do not exploit Gartner's
construction here, our points do lie in a theoretical setting which is also covered by him and therefore the natural question arises
of whether Gartner's points are equal to ours when both constructions are available. We address this issue in \S
\ref{ComparingGartner}, where we point out that Conjecture (BSD) implies that one is a non-zero multiple of the other;
the difference between them is that ours are numerically accessible, and this stands as the main motivation of this
article.

This section is devoted to review the work of Darmon and Gartner, settling on the way the notations that
shall be in force for the rest of this note. As Gartner's exposition \cite{Ga}, \cite{Gartner-article} is already an
excellent  account of the theory, we choose here to reword it in the classical language of Hilbert modular forms, under
the simplifying hypothesis that the narrow class number of the base field $F_0$ is $1$.

In doing so, we take the chance to contribute to the theory with a few novel aspects. To name one, it will be convenient for our purposes to work with the natural, relatively straight-forward extension of the theory
to the setting of eigenforms with not necessarily trivial nebentypus and whose eigenvalues generate a number field of
arbitrarily large degree over $\Q$. This will lead us to a construction of rational points on higher-dimensional modular
abelian varieties of $\GL_2$-type.

\subsection{Quadratic extensions and $L$-functions}\label{sec:quadratic}

Let $F_0\subset \R$ be a totally real number field, together with a fixed embedding into the field of real numbers. Write $d=[F_0:\Q]$ for its degree over $\Q$ and let $R_0\subset F_0$ denote its ring of
integers. In order to keep our notations simple, we assume that the narrow class number of $F_0$ is $1$.

Let $N$ be a square-free integral
ideal of $F_0$ and let $\psi$ be a Hecke character of conductor $N$.
Let $f_0\in S_2(N,\psi)$ be a normalized Hilbert eigenform of parallel weight $2$, level $N$ and nebentypus $\psi$. Let
$\Q_{f_0}$ denote the number field generated by the eigenvalues of the Hecke operators acting on $f_0$, which we regard
as embedded in the algebraic closure $\bar\Q$ of $\Q$ in the field $\C$ of complex numbers; for each $\sigma \in
\Hom(\Q_{f_0},\bar\Q)$, there exists a unique normalized eigenform ${}^{\sigma}f_0$ whose family of eigenvalues is equal
to the family of eigenvalues of $f_0$ conjugated by $\sigma$.

The following standard conjecture is a generalized form of the
Eichler--Shimura philosophy:

\begin{conjecture}\label{ass1}
There exists an abelian variety $A=A_{f_0}/F_0$ of dimension $g=[\Q_{f_0}:\Q]$ and conductor $N^g$ such that $\Q\otimes
\End_{F_0}(A)\simeq \Q_{f_0}$, and whose L-series factors as
\begin{equation}\label{Lf}
L(A,s) = \prod_{\sigma \in \Hom(\Q_{f_0},\bar\Q)} L({}^{\sigma}f_0,s).
\end{equation}
\end{conjecture}

Note that, if such an $A$ exists, it is well-defined only up to isogenies.

Conjecture \ref{ass1} is known to hold when (JL) is satisfied.  When
(JL) fails it is not even known whether there exists a motive $M_{f_0}$ over $F$ whose L-function is \eqref{Lf} and one
certainly does not expect the motive $h^1(E)$ to arise in the cohomology of any (quaternionic) Hilbert variety
(cf.\,\cite{BlRo} and  for more details). See \cite{De} for the numerical verification of Conjecture \ref{ass1} in several instances in which (JL) fails.

We shall {\em assume for the remainder of this section} that Conjecture \ref{ass1} holds true.

\vspace{0.3cm}

Let $K/F_0$ be a quadratic extension such that $(\disc(K/F_0),N)=1$ and $r_2(K/F_0)\geq 1$. Label the set of embeddings
of $F_0$ into the field $\R$ of real numbers as
$$
\{ v_1,v_2,...,v_r,v_{r+1},...,v_{d}: F_0\hookrightarrow \R\}, \quad 1\leq r \leq d
$$ in such a way that
\begin{itemize}
\item $v_1$ is the embedding fixed at the outset that we use to identify $F_0$ as a subfield of $\R$,
\item each of the places $v_2,...,v_r$ extends to a pair of real places of $K$, which by a slight abuse of notation we denote $v_j$ and
$v_j'$ for each $j=2,...,r$, and
\item each of the places $v_1, v_{r+1}, ...,v_{d}$ extends to a complex place on $K$, that we still denote with the same letter; we use $v_1$ to regard $K$ as a subfield of $\C$.
\end{itemize}

\begin{definition}\label{defATC}
If $r=1$, the set $\{ v_2,...,v_r \}$ is empty and $K/F_0$ is a CM-field
extension.

If $r=2$ we call $K/F_0$ an almost totally complex (ATC) extension.

If $r=d$ we have $\{v_1,
v_{r+1}, ...,v_{d}\}=\{v_1\}$ and we say that $K/F_0$ is almost totally real (ATR).
\end{definition}

Letting $\varepsilon_K$ denote the quadratic Hecke character of $F_0$ associated with the extension $K/F_0$, the
L-function of the base change of $A$ to $K$ is
$$
L(A/K,s) = L(A,s) \cdot L(A,\varepsilon_K, s) = \prod_{\sigma \in \Hom(\Q_{f_0},\bar\Q)} L({}^{\sigma}f_0,s)\cdot
L({}^{\sigma}f_0,\varepsilon_K,s).
$$
It extends to an entire function on $\C$ and satisfies a functional equation relating the values at $s$ with $2-s$.
Assume  that the sign of the functional equation of $L(f_0/K,s)=L(f_0,s)\cdot
L(f_0,\varepsilon_K,s)$ is $-1$. This is equivalent to saying that the set
\begin{equation}\label{ram}
\{v_{r+1},...,v_d\} \cup \{ \wp \mid N, \wp \mbox{ inert in } K\}
\end{equation}
has {\em even} cardinality.

Let $B$ be the (unique, up to isomorphism) quaternion algebra over $F_0$ whose set of places of ramification is
$\mathrm{Ram}(B)=$\eqref{ram}. In particular we have $B\otimes_{F_0,v_j} \R \simeq \M_2(\R)$ for $j=1,...,r$, and the choice of such isomorphisms gives rise to an embedding
\begin{equation}\label{embv_j}
(v_1,...,v_r): B^\times \hookrightarrow \GL_2(\R)\times \overset{(r)}{...}\times \GL_2(\R)\subset (B\otimes_\Q
\R)^\times.
\end{equation}

 Let $N^+$ be the product of primes in $F_0$ such that divide $N$ and are split in $K$, and $N^-$ the
product of primes that divide $N$ and remain inert in $K$. Choose an Eichler order $\mathcal O$ of level $N^+$ in
$B$ together with, for each prime $\wp\mid N^+$,
isomorphisms $i_{\wp}: B\otimes F_{0,\wp}\simeq M_2(F_{0,\wp})$ such that
$$
i_\wp(\cO) = \{ \begin{smallmat}
abcd
\end{smallmat}, \wp \mid c \}\subseteq \M_2(R_{0,\wp}).
$$

\begin{definition} Let $F_0^+$ denote the subgroup of $F_0^\times$ of totally positive elements and $B^+$ be the subgroup of elements in $B^\times$ whose reduced norm lies in $F_0^+$. Define the congruence subgroups
$$
\Gamma_0=\Gamma_0^{N^-}(N^+) = \cO^\times\cap B^+\quad \mbox{and}$$ $$\Gamma_1 =\Gamma_1^{N^-}(N^+) = \{
\gamma \in \Gamma_0, i_\wp(\gamma) \cong \begin{smallmat}
1\star01
\end{smallmat}, \wp\mid N^+\} \subset \Gamma_0.
$$
\end{definition}

Through \eqref{embv_j}, $\Gamma_1$ acts on the cartesian product $\mathcal H^r=\mathcal H_1\times ...\times \mathcal
H_r$ of $r$ copies of
Poincar\'e's upper-half plane and we let $X_\C = \Gamma_1\backslash \mathcal H^r$ denote its quotient, which has a
natural structure of analytic manifold with finitely many isolated singularities.

\begin{definition}
Let $F_0^{\mathrm{gal}}$ denote the galois closure of $F_0$ in $\C$ and
view the places $v_i$ as elements of the Galois group $G=\Gal(F_0^{\mathrm{gal}}/\Q)$, so that $v_1=\mathrm{Id}$. The {\em reflex field of $B$} is the subfield $F_0^\star$ of $F_0^{\mathrm{gal}}$ fixed by the subgroup of those $\sigma \in G$ such that
$\sigma \cdot \{v_{1},...,v_r\} = \{v_{1},...,v_r\}$.
\end{definition}

The cases one encounters most often in the literature arise when either $r=1$, where $F_0^\star=F_0$, or when $r=d$, in which case $F_0^\star=\Q$.

Let $$X=X_1^{N^-}(N^+)/F_0^\star$$ denote
Shimura's canonical model over $F_0^\star$ of $X_\C$, as introduced e.g. in \cite[\S 12]{Milne}. If $\mathrm{Ram}(B)\ne \emptyset$, $X_\C$ is compact and $X$ is projective over $F_0^\star$, while if $\mathrm{Ram}(B)=\emptyset$ then $B=\M_2(F_0)$ and $X_\C$ admits a canonical compactification by adding a finite number of cusps; by an abuse of notation, we continue to denote $X$ the resulting projective model.

\subsection{Oda-Shioda's conjecture}\label{sec:OdaShioda}

Let $\Sigma = \{ \pm 1\}^{r-1}$ and for each $\epsilon=(\epsilon_2,...,\epsilon_r)\in \Sigma$, let $\gamma_\epsilon \in
\cO^\times$ be an element such that $v_j(n(\gamma_\epsilon)) = \det(v_j(\gamma_\epsilon))>0$ if $j=1$ or $\epsilon_j=+1$, and
$v_j(n(\gamma_\epsilon))<0$ if $\epsilon_j=-1$. Such elements exist thanks to our running assumption that the narrow
class number of $F_0$ is $1$. For $\tau_j\in \mathcal H_j$, set
$$
\tau_j^{\epsilon} = \begin{cases}
v_j(\gamma_\epsilon) \tau_j \mbox{ if } j=1 \mbox{ or } \epsilon_j=+1, \\
v_j(\gamma_\epsilon) \bar{\tau}_j \mbox{ if } \epsilon_j=-1.
\end{cases}
$$

For each $0\leq i\leq 2r$, let $H_i(X_\C, \Z) = Z_i(X_\C,\Z)/B_i(X_\C,\Z)$ denote the $i$-th Betti homology group of
$X_\C$. Attached to $f_0$ there is the natural holomorphic $r$-form on $\mathcal H^r$ given by
$$
\omega_{f_0} =  (2\pi i)^r f_0(\tau_1,...,\tau_r) d\tau_1 ... d\tau_r,
$$
which is easily shown to be $\Gamma_1$-invariant (and to extend to a smooth form on the cusps, if $B=\M_2(F_0)$), giving
rise to a regular differential $r$-form $\omega_{f_0}\in H^0(X_\C, \Omega^r)$.

Label the set $\Hom(\Q_{f_0},\C) = \{ \sigma_1, ..., \sigma_g\}$ of embeddings of $\Q_{f_0}$ into the field of complex
numbers. The set $\{ \sigma_1(\omega_{f_0}), ..., \sigma_g(\omega_{f_0})\} $ is then a basis of the $f_0$-isotypical component
of $H^0(X, \Omega^r)$.

\begin{definition}\cite[(8.2)]{Drpmec}, \cite[\S 2]{Ga}\label{def-omegabeta} Let $d_0$ be a totally positive generator of the different ideal
of $F_0$ and let $\beta\colon \Sigma\ra \{\pm 1\}$ be a character. The differential $r$-form $\omega^\beta_{f_0}$ on $X$
associated with $f_0$ and $\beta$ is
$$
 \omega_{f_0}^\beta := |d_0|^{-1/2} (2\pi i)^r\sum_{\epsilon \in \Sigma}
\beta(\epsilon)f_0(\tau_1^{\epsilon},...,\tau_r^\epsilon)
d\tau_1^\epsilon ... d\tau_r^\epsilon.
$$
\end{definition}

If $r=1$, note that the only choice for $\beta$ is the trivial one and in this case one recovers the usual holomorphic $1$-form $\omega_{f_0}$ on the Shimura curve $X/F_0$. On the other hand, when $r>1$, the differential form $\omega_{f_0}^\beta$ is not holomorphic anymore for any choice of $\beta$, including the trivial one.

\begin{definition}
The {\em lattice of periods} of $\omega_{f_0}^\beta$ is
$$
\Lambda_{f_0}^\beta = \{ (\int_{\tilde Z} \sigma_1(\omega_{f_0}^\beta), ..., \int_{\tilde Z} \sigma_g(\omega_{f_0}^\beta)), \, \tilde Z
\in H_r(X_\C,\Z)\} \subseteq  \C^g.
$$
\end{definition}

In addition to that, under the running assumption that conjecture \ref{ass1} holds true, we can also introduce another
lattice as follows. For each $j=1,...,r$, let $A_j=A\times_{F_0,v_j} \C$  denote the base change of $A$ to the field of
complex numbers via the embedding $F_0\overset{v_j}{\hookrightarrow} \R \subset \C$. Since we identify
$v_1$ with the identity embedding, $A_1$ is identified with $A$. Let $H_1(A_j,\Z)^{\pm}$ be the
$\Z$-submodule of $H_1(A_j,\Z)$ on which complex conjugation acts as $+1$ (resp. $-1$). Since $\Q_{f_0}\simeq \Q\otimes
\End_{F_0}(A)$, there is a natural action of $\Q_{f_0}$ on $H_1(A_j,\Q)^\pm$ and in fact the latter is a free module of
rank $1$ over the former.

Similarly, the space $H^0(A,\Omega^1)$ of global regular differential 1-forms on $A$ is an $F_0$-vector space of
dimension $g$ equipped with a $F_0$-linear action of $\Q_{f_0}$ inherited from the isomorphism $\Q_{f_0}\simeq \Q\otimes
\End_{F_0}(A)$.

Recall that $R_0$ stands for the ring of integers of $F_0$. Make the following choices:
\begin{itemize}
\item A regular differential $\omega_A \in H^0(A,\Omega^1)$ which extends to a smooth differential on the N\'eron model
of $A$ over $R_0$ and generates $H^0(A,\Omega^1)$ as a $\Q_{f_0}$-module.
\item For each $j=1,...,r$, generators $c^+_j$, $c^-_j$ of $H_1(A_j,\Q)^+$ and $H_1(A_j,\Q)^-$ as $\Q_{f_0}$-modules.
\end{itemize}

Given these choices, define
$$
\Omega_j^+ = \int_{c_j^+} v_j(\omega_A)\in \C, \quad \Omega_j^- = \int_{c_j^-}
v_j(\omega_A)\in \C, \ \mbox{for} \ \, j=1,...,r \quad \mbox{ and }$$
$$
\Omega_{\beta}=\Omega_2^{\beta_2(-1)}\cdot ...\cdot \Omega_r^{\beta_r(-1)}.
$$

\begin{definition}
Let $R_{f_0}$ denote the ring of integers of $\Q_{f_0}$ and define
$$
\Lambda_0^\beta := \Omega_{\beta} \cdot (\Z \Omega_1^+ +
\Z\Omega_1^-)\subset \C,\quad \Lambda^\beta_A := \Lambda_0^\beta\otimes_\Z
R_{f_0} \subseteq \C\otimes_\Z \Q_{f_0}\simeq \C^g.
$$
\end{definition}

Let us now analyze how these lattices depend on the above choices. Note that $\omega_A$ is well-defined only up to
multiplication by units $u\in R_0^\times$ and non-zero endomorphisms $t\in \Q_{f_0}^\times$. If we replace $\omega_A$ by
$u\cdot \omega_A$, we obtain
$$
\Lambda_0^\beta (u\cdot \omega_A) = \langle \int_{c_1^+} v_1(u\cdot
\omega_A)\cdot\prod_{j=2}^r \int_{c_j^{\beta_j(-1)}} v_j(u\cdot \omega_A),\int_{c_1^-} v_1(u\cdot
\omega_A)\cdot\prod_{j=2}^r \int_{c_j^{\beta_j(-1)}} v_j(u\cdot \omega_A)\rangle =
$$
$$
= \mathrm{N}_{F_0/\Q}(u) \Lambda_0^\beta (\omega_A) = \Lambda_0^\beta (\omega_A),
$$
because $\mathrm{N}_{F_0/\Q}(u)=\pm 1$, and thus also $\Lambda^\beta_A(u\cdot \omega_A) =\Lambda^\beta_A(\omega_A)$.

If instead we replace $\omega_A$ by $t\cdot \omega_A$ for some $t\in \Q_{f_0}^\times$, then
$$
\Lambda^\beta_A(t \omega_A)=\{ \Omega_{\beta} \int_{c_1^+} v_1(t^*\omega_A)\otimes
s,\Omega_{\beta} \int_{c_1^-} v_1(t^*\omega_A)\otimes s, s\in
R_{f_0}\} =
$$
$$
= \{ \Omega_{\beta} \int_{c_1^+} v_1(\omega_A)\otimes s t,
\Omega_{\beta} \int_{c_1^-} v_1(\omega_A)\otimes s t, s\in R_{f_0}\}
$$
and therefore $\Q\otimes_\Z\Lambda^\beta_A(t \omega_A) =\Q\otimes_\Z\Lambda^\beta_A(\omega_A)$. We reach to the same
conclusion if we take different choices of homotopically equivalent paths $c_j^+$ or $c_j^-$. Hence the $\Q$-submodule
$\Q\otimes_\Z\Lambda_A^\beta$ of $\C^g$ is determined uniquely independently of the choices made.

\begin{conjecture}[Oda, Yoshida]\cite{Od}\label{Oda-conj}
The lattices $\Lambda_{f_0}^\beta$ and $\Lambda_A^\beta$ are commensurable, that is to say, $$\Q\otimes_\Z
\Lambda_{f_0}^\beta=\Q\otimes_\Z \Lambda_A^\beta,$$
and therefore there exists an isogeny of abelian varieties
$$
\eta_\beta: \C^g/\Lambda_{f_0}^\beta \overset{\sim}{\lra }  \C^g/\Lambda_A^\beta \simeq A(\C)=\C^g/\Lambda_1,
$$
where the last isomorphism is given by multiplication by $\Omega_{\beta}^{-1}$.
\end{conjecture}

Note that, consistently with Conjecture \ref{ass1}, the above Conjecture \ref{Oda-conj} only concerns the isogeny class of the abelian variety $A$.

\begin{remark}
If $r=1$ and (JL) holds, Conjecture \ref{Oda-conj} holds true: the abelian variety $A$ may be constructed explicitly as a constituent of the Jacobian of the Shimura curve $X$ and it follows from the very construction that the lattices $\Lambda_{f_0}^\beta$ and $\Lambda_A^\beta$ are commensurable.\end{remark}

\subsection{Darmon points}\label{sec:Darmonpoints}
Let $\mathcal{Z}_{r-1}(X_\C)$ denote the set of null-homologous cycles of real dimension $r-1$ in $X_\C$.
For each character $\beta$ as above, Conjecture \ref{Oda-conj} allows us to define the {\em topological Abel-Jacobi map}
\begin{equation}\label{AJtop}
\begin{matrix}
\mathrm{AJ}^\beta: & \mathcal{Z}_{r-1}(X_\C) &  \lra &  A(\C)\\
   & T & \mapsto & \eta_\beta\left(\int_{\tilde T } \omega_{f_0}^\beta\right),
\end{matrix}
\end{equation}
where  $\tilde T\in
C_r(X_\C,\Z)$ is any $r$-dimensional chain satisfying $\partial \tilde T=T$. Observe that $\tilde T$ is determined up to
elements in $H_r(X_\C,\Z)$, so that the quantity $\int_{\tilde T } \omega_{f_0}^\beta\in \C\otimes_\Z \Q_{f_0}$ is a
well-defined element in $\C^g/\Lambda_{f_0}^\beta$ and $\mathrm{AJ}^\beta$ is thus a well-defined map.

\begin{remark}
If $r=1$ and (JL) holds, the map $\mathrm{AJ}^\beta$ is nothing else but the classical {\em algebraic} Abel-Jacobi map
of curves $X_\C \lra
A(\C)$, which factors through the jacobian of $X_\C$. This was one of Darmon's motivations for extending the rule to
the general case, though the reader must be warned that when $r>1$ the maps $\mathrm{AJ}^\beta$ are {\em not}
algebraic.\end{remark}

Let now $c\subseteq R_0$ be an integral ideal of $F_0$ relatively coprime with $N$ and let $R_c := R_0 + cR_K\subseteq
R_K$ be the order of conductor $c$ in the ring of integers of $K$. Let $\eta$ be the homomorphism
\[
 \eta\colon \cO \lra R_0/N^+R_0
\]
sending an element $x\in \cO$ to the upper-left hand entry of its image in $\cO\otimes_{R_0} R_{0,N^+}\simeq \M_2(R_{0,N^+})$, taken modulo
$N^+R_{0,N^+}$.

\begin{definition} Fix a factorization of ideals $N^+ R_K= \mathfrak N^+ \cdot \bar{\mathfrak N}^+$.
An  embedding  of $R_0$-algebras $\varphi\colon R_c\hookrightarrow \cO$ is said to be \emph{optimal} if
$\varphi(R_c)=\varphi(K)\cap \cO$. We say that $\varphi$ is
\emph{normalized} (with respect to $\mathfrak N^+$) if it satisfies the following conditions:
\begin{enumerate}
 \item $\varphi$ acts on $u=(\tau_1,1)\in \C^2$ as $\varphi(a)_1\cdot
u=v_1(a)\cdot u$ for all
$a\in R_c$, where $\varphi(a)_1$ denotes the image of $\varphi(a)$ in $B\otimes_{F_0,v_1}\R$.
\item The kernel of $\eta\circ \varphi$ is equal to $\mathfrak N^+$.
\end{enumerate}
We denote by $\mathcal E(R_c,\mathcal O)$ the set of normalized optimal embeddings.
\end{definition}

Recall that $v_1$ extends to a complex place of $K$ and that $v_2,\dots,v_r$ extend to real places. Given  $\varphi\in
\mathcal E(R_c,\mathcal O)$, the action of $K^\times$ on $\C$ by fractional linear transformations
via the composition
of $\varphi$ and the isomorphism $(B\otimes_{F_0,v_1}\R)^\times\simeq \GL_2(\R)$ has a unique fixed point $z_1\in
\mathcal H_1$.
For $j=2,\dots,r$ it has  two fixed points $\tau_j,\tau_j'\in \R=\partial \mathcal H_j$ under the isomorphism
$(B\otimes_{F_0,v_j}\R)^\times\simeq \GL_2(\R)$. Let $\gamma_j$ be the geodesic path joining $\tau_j$ and $\tau_j'$ in
$\mathcal H_j$.

\begin{definition}
We denote by $T_\varphi$ the $(r-1)$-real dimensional cycle
in $X_\C$ given by the image of the region
$$R_\varphi=\{z_1\}\times\gamma_2\times\cdots\times\gamma_r\subset \mathcal H_1\times
\stackrel{r)}{\cdots}\times \mathcal H_r$$
under the natural projection map $\mathcal H^r\lra X_\C$.
\end{definition}

Note that the stabilizer of $R_\varphi$ in $\Gamma_1$ is the subgroup $\Gamma_\varphi=\varphi(K)\cap \Gamma_1$ and
therefore there is a natural homeomorphism $T_\varphi\simeq \Gamma_\varphi \backslash R_\varphi$. As an application of
the Matsushima--Shimura Theorem \cite{MaSh}, it is easy to show (cf.\,\cite[Proposition 4.3.1]{Gartner-article}) that
the class of $T_\varphi$ has finite order in $H_{r-1}(X_\C,\Z)$. In particular, if $e$ denotes the order of
$T_\varphi$ then $eT_\varphi$ is null-homologous. This allows the following definition.

\begin{definition}\label{Darmon-point} The \emph{Darmon point} attached to $\varphi$ and $\beta$ is
$$P^\beta_\varphi:= \frac{1}{e}\mathrm{AJ}^\beta(eT_\varphi) \in A_1(\C).$$
\end{definition}

Darmon points are conjectured to be rational over certain number fields, with the Galois action given by an
explicit reciprocity law. This is the content of Conjecture \ref{ATR-gen}. Next we define the number
fields and the actions involved in the conjecture.

Define
$$
U_c: = {\hat R_0}^\times (1+c{\hat R_K}) \subset \hat{K}^\times=(K\otimes_\Z \hat \Z)^\times.
$$
For every character $\beta$ of $\Sigma$, let $H_c^\beta$ denote the abelian  extension of $K$ corresponding by class
field theory to the open compact subgroup $K_\infty^\beta \times U_c$, where
\begin{equation}\label{eq: def open compact subgroup}
K_\infty^\beta:=\prod_{j=1,r+1,...,d} \C^\times \times \prod_{\substack{j=2,...,r\\ \beta_j(-1)=+1}} \R^\times
\times\prod_{\substack{j=2,...,r\\\beta_j(-1)=-1}}\R_{>0}.
\end{equation}

Recall that $\psi$ denotes the Nebentypus of $f_0$. For $\varphi\colon
R_c\hookrightarrow \cO$ a normalized optimal embedding define
\[
 U_\varphi^+=\mathrm{ker}(U_c \lra (R_K/\mathfrak N^+R_K)^\times\simeq (R_0/N^+R_0)^\times \overset{\psi}{\ra}
\C^\times ),
\]
\[
 U_\varphi^-=\mathrm{ker}(U_c \lra (R_K/\bar{\mathfrak N}^+R_K)^\times \simeq (R_0/N^+R_0)^\times \overset{\psi}{\ra}
\C^\times  ).
\]
Denote by $L_\varphi^\beta/H_c^\beta$ (resp. ${L'}_\varphi^\beta/H_c^\beta$)  the abelian extension of the ring
class field of conductor $c$ associated to $K_\infty^\beta\times U_\varphi^+$ (resp. $U_\varphi^-\times
K_\infty^\beta$). Let also $U_\varphi=U_\varphi^+\cap U_\varphi^-$
and  let $H^\beta_\varphi/H^\beta_c$ be the extension associated to $K_\infty^\beta \times U_\varphi
\subset \A_K^\times$.

Observe that we can extend $\psi$ to a character on $\cO^\times$ by composing with $\eta$. Then we define
\[
\Gamma_1 \subseteq  \Gamma_\psi:=\Gamma_\psi^{N^-}(N^+):=\{\gamma\in \Gamma_0\ \colon \psi(\eta \gamma)=1\} \subseteq
\Gamma_0.
\]

The  group $\Gamma_0$ acts on $\mathcal E(R_c,\cO)$ by conjugation, and we denote by $\mathcal E(R_c,\cO)/\Gamma_0$ the
set of conjugacy classes. Any element $W_\varepsilon \in\Gamma_0\setminus \Gamma_\psi$ defines an involution on
$\mathcal E(R_c,\cO)/\Gamma_0$ which interchanges the preimages of the natural projection $\mathcal
E(R_c,\cO)/\Gamma_\psi\ra \mathcal E(R_c,\cO)/\Gamma_0$. In addition to $W_\varepsilon$, there is also an Atkin--Lehner
involution  acting naturally on the set of embeddings, although it does not preserve the normalization. To be more
precise, let $\omega_N$ be an element in $B$ such that
\begin{itemize}
\item for every $\wp\mid N^-$, $\omega_N$ generates the single two-sided ideal of $\cO\otimes R_{0,\wp}$ of norm $\wp$, and

\item for every $\wp\mid N^+$, $\iota_\wp(\omega_N)=\left( \begin{array}{cc} 0 & -1 \\ \pi_\wp & 0  \end{array}\right)$, where $\pi_\wp$ is any
uniformizer in $R_{0,\wp}$.

\end{itemize}
Let us denote by  $\bar{\mathcal{E}}(R_c,\cO)$ the set of optimal embeddings normalized with respect to
$\bar{\mathfrak{N}}^+$. Then the map $\varphi\mapsto W_N(\varphi):=\omega_N\varphi\omega_N^{-1}$ is a bijection between
${\mathcal{E}}(R_c,\cO)$ and $\bar{\mathcal{E}}(R_c,\cO)$. From now on  denote by $W_N(P_\varphi^\beta)$ the point
$P_{W_N(\varphi)}^\beta$.

Finally, there is also a natural action of $\hat K^\times$  on $\mathcal E(R_c,\cO)$, which works as follows. Pick a
finite id\`ele
$x\in \hat K^\times$ and an embedding $\varphi$ in $\mathcal E(R_c,\cO)$.  Since the class number of $\cO$ is
$h(\cO)=h(F_0^*)=1$ by \cite[Cor.\,5.7 bis]{Vigneras},
the fractional ideal $I_x=\varphi(x)\hat\cO\cap B$ is principal, generated by some $\gamma_{x}\in B^\times$ with
$n(\gamma_x)\in F_0^+$. Moreover, we can choose $\gamma_x$
such that $a_x=\varphi(x_{\mathfrak N^+}x_{\overline{\mathfrak N}^+})^{-1}\cdot
\gamma_x$ lies in the kernel of $\psi \eta$. (Indeed, note that, locally at the primes $\wp\mid N^+$, we have
$\varphi(x_{\mathfrak N^+}x_{\overline{\mathfrak N}^+})^{-1} I_{x,\wp} = \cO_\wp$ and thus $a_x$ belongs to
$\cO_\wp^\times$. It hence makes sense to consider its image under $\psi \eta$.  We can assume $\gamma_x$ is as claimed
by replacing it by a suitable unit in $\cO^\times$.)  We define $x\star
\varphi:=\gamma_{x}^{-1}\circ\varphi\circ \gamma_{x}$. Observe that $U_\varphi^+$
acts trivially on $\mathcal E(R_c,\cO)$.

For $y\in K_\infty^\times$ and a character $\beta \colon \Sigma\ra \{\pm 1\}$, set
\[
 \beta(y)=\prod_{j=2}^r \beta(\sign(\prod_{w|v_j} y_{w})).
\]

The following statement collects, in a precise form, the conjectures (SH1), (SH2), (SH3) that were somewhat vaguely formulated in the introduction for Darmon points over abelian extensions of $K$.

\begin{conjecture}\label{ATR-gen}
\begin{enumerate}
\item If $\varphi\in \mathcal
E(R_c,\cO)$ then $P^\beta_\varphi \overset{?}{\in } A_1(L^\beta_\varphi).$

\item For any character $\chi: \Gal(L^\beta_\varphi/K)\ra \C^\times$, the point  $$P_\chi:= \sum_{\sigma \in
\Gal(L^\beta_\varphi/K)} \chi(\sigma)^{-1} \sigma(P^\beta_\varphi) \in A_1(L^\beta_\varphi)\otimes_{\Z} \C$$ is non-zero
if and only if $L'(f_0/K,\chi,1)\ne 0$.

\item For any $a=(a_\infty,a_f)\in
\A_{K}^\times$
we
have that $\mathrm{rec}(a)P_\varphi^\beta=\beta(a_\infty)P_{a_f \star \varphi}$. In addition, for any $\tau\in
\Gal(H_\varphi^\beta/F_0)$ whose restriction to $K$ is not trivial, there exists an element $\sigma\in
\Gal(H_\varphi^\beta/K)$ such that
\[
 \tau(P_\varphi^\beta)=W_N(\sigma(P_\varphi^\beta)) \ \ (\mathrm{mod} \ A_1(H_\varphi^\beta)_{\mathrm{tors}}).
\]

\end{enumerate}
\end{conjecture}

Here $\mathrm{rec}: \A_{K}^\times \lra \Gal(K^{\mathrm{ab}}/K)$ is Artin's reciprocity map, normalized so that
uniformizers  correspond to geometric Frobenius elements. Note that the three statements of Conjecture \ref{ATR-gen} are
the translation to the current context of (SH1), (SH2), (SH3) given in the introduction.

%
%
% The following two propositions are wrong...the two sets do not have the same cardinal if $B$ is not a matrix algebra.
% Probably there are some factors $2$ missing...
%
%

% \begin{proposition}
%  There is a natural bijection between $\Gal(H_c^\beta/K)=\A_K^\times/(K^\times U_c K_\infty^\beta)$ and
% $\mathcal
% E(R_c,\cO)/\Gamma_0^\beta$, where
% \[
%  \Gamma_0^\beta=\{\gamma\in \cO^\times \colon v_i(n(\gamma))>0\ \text{ if } i=1 \ \text{ or } \ \beta(\epsilon_j)=-1\}.
% \]
% \end{proposition}
% % \begin{proof}
%  TO DO. Again  I am assuming $F_0^\star=\Q$. However, observe that in the case where $F_0^\star=\Q$, if $\beta=1$ this
% particularizes to Darmon's Lemma 8.16 of \cite{Drpmec}...OK and maybe this is only valid for $M=\M_2(F_0)$...what if
% $B$ is division?
% \end{proof}
%
% \begin{proposition}
%  There is a natural bijection between $\Gal(L_\varphi^\beta/K)=\A_K^\times/(K^\times U_\varphi^+K_\infty^\beta)$ and
% $\mathcal
% E(R_c,\cO)/\Gamma_\psi^\beta$, where
% \[
%  \Gamma_\psi^\beta=\{\gamma\in \Gamma_0^\beta \, \colon\, \psi(\gamma)=1 \}.
%  \]
%  \end{proposition}
% \begin{proof}
%  TO DO, again we are in the case $F_0^\star=\Q$ and $B$ not division...
% \end{proof}

\subsection{Darmon-Logan's algorithm for the computation of ATR points}\label{subsection: Darmon's algorithm for ATR points}

One naturally wonders whether Darmon points, as introduced in Definition \ref{Darmon-point}, can be computed effectively
in explicit examples.  A positive answer would allow us to test Conjecture \ref{ATR-gen} numerically, leading to an
explicit construction of rational points on elliptic curves over number fields which were not accessible before.

However, the image of $T_\varphi$ under the Abel-Jacobi map $\mathrm{AJ}^\beta$ of \eqref{AJtop} can only be computed
provided we are able to write down  an explicit candidate for a region $\tilde{T}_\varphi$ having $T_\varphi$ as
boundary and we can integrate it against the differential form $\omega_{f_0}^\beta$. The latter only seems possible when
there is available a natural, explicit description of $\omega_{f_0}^\beta$. And this is precisely the case when the
following Gross-Zagier assumption holds:

\begin{ass}\label{GZass} $r=d$ and all the primes dividing $N$ are split in $K/F_0$.\end{ass}

Indeed, when this is the case we have that $K/F_0$ is an ATR extension, $B\simeq \M_2(F_0)$ and $X$ is a $d$-dimensional
 Hilbert modular variety over $F_0^\star=\Q$. In addition, and most importantly, the form $\omega_{f_0}^\beta$ admits a
natural fourier expansion around the cusp at infinity, and there exist algorithms which allow to compute it up to a
given precision: cf.\,e.g.\,\cite{DeVo}.

If this hypothesis does not hold true, we are at a loss to compute numerical approximations to the points
$P_\varphi^\beta$. We impose Assumption \ref{GZass} for the remainder of this section, that we devote to sketch
Darmon-Logan's algorithm for computing an explicit chain $\tilde{T}_\varphi$ whose boundary is $T_\varphi$. We adapt it
to our slightly more general setting in which $[\Q_{f_0}:\Q]\geq 1$, so that we can also make use of it later. To
simplify the exposition, and since this is the case
encountered in the numerical example described in \S \ref{subsection: numerical examples}, let us assume also that $[F_0\colon \Q]=2$.

 The key point in Darmon--Logan's approach is the definition of certain
$3$-limit integrals
of $\omega_{f_0}^\beta$, allowed by the following interpretation of
the homology groups of $X$.   Let $\Gamma$ denote the quotient of $\Gamma_\psi$ by the normal closure of the subgroup
generated by the elliptic and parabolic elements.  Let $I_{\Gamma}$ be the augmentation ideal, which sits in the exact
sequence
\begin{equation*}
%\label{eq: augmentation ideal}
 0\lra I_{\Gamma}\lra \Z[{\Gamma}] \lra \Z \lra 0.
\end{equation*}
For a $\Gamma$-module $M$ we denote by $M_\Gamma=M/I_\Gamma M$ its ring of $\Gamma$-coinvariants.  Tensoring the above
sequence  by $I_\Gamma$ and taking the group homology exact sequence we obtain
\begin{equation}\label{eq: tensored exact sequence}
 0\lra H_1(\Gamma, I_\Gamma)\lra (I_\Gamma\otimes_\Z I_\Gamma)_\Gamma \stackrel{\partial}{ \lra} (Z[\Gamma]\otimes_\Z
I_\Gamma)_\Gamma \lra (I_\Gamma)_\Gamma\lra 0,
\end{equation}
where $\partial$ is the natural map induced by the inclusion $I_\Gamma\subset Z[\Gamma]$. There are canonical
isomorphisms $(I_\Gamma)_\Gamma\simeq H_1(\Gamma,\Z)$ and $H_1(\Gamma,I_\Gamma)\simeq H_2(\Gamma,\Z)$. Therefore, in
view of the natural isomorphisms $H_1(\Gamma,\Z)\simeq H_1(X,\Z)$ and $H_2(\Gamma,\Z)\simeq H_2(X,\Z)$ one can identify
\eqref{eq: tensored exact sequence} with the exact sequence
\begin{equation}\label{eq: topological tensored exact sequence}
 0\lra Z_2(X, \Z)\lra C_2(X,\Z) \stackrel{\partial}{ \lra} Z_1(X,\Z) \lra H_1(X,\Z)\lra 0,
\end{equation}
where $\delta$ is the topological boundary map.

Recall that   integrals of $\omega_{f_0}^\beta$ satisfy the following invariance property:
\[
 \int_x^y\int_z^t \omega_{f_0}^\beta=\int_{\gamma x}^{\gamma y}\int_{\gamma z}^{\gamma t}\omega_{f_0}^\beta, \ \ \
\text{for all} \ \gamma\in \Gamma \text{ and } x,y,z,t\in\mathcal H.
\]
We remark that in this expression $\gamma$ is acting on the outer limits (resp. inner limits) of the integral through
$v_1$ (resp. $v_2$). By choosing
base points $z_1\in\mathcal H_1$ and $z_2\in \mathcal H_2$ one obtains then a group homomorphism
\[
 \begin{array}{cccc}
  I_{z_1,z_2}^\beta\colon &(I_\Gamma\otimes_\Z I_\Gamma)_\Gamma & \lra & \C\\
 & (\gamma_1-1)\otimes (\gamma_2-1) & \longmapsto & \int_{z_1}^{\gamma_1 z_1}\int_{z_2}^{\gamma_2
z_2}\omega_{f_0}^\beta,
 \end{array}
\]
which can be identified with the map
\[
 \begin{array}{ccc}
  C_2(X,\Z) & \lra & \C\\
  T & \longmapsto & \int_T\omega_{f_0}^\beta.
 \end{array}
\]
Observe that the identification $H_1(\Gamma,I_\Gamma)\simeq Z_2(X,\Z)$ yields then an explicit description
of the lattice $\Lambda_{f_0}^\beta$;
indeed $\Lambda_{f_0}^\beta\simeq I^\beta_{z_1,z_2}(H_1(\Gamma,I_\Gamma))$.

Suppose now that $1\otimes (\gamma_2-1)\in \Z[\Gamma]\otimes I_\Gamma$ is such that $e (1\otimes (\gamma_2-1))$ lies in
the image of $\delta$ for some integer $e$. That would correspond in \eqref{eq: topological tensored exact sequence} to
a cycle $T$ such that $eT$ is null homologous. Following \cite{DL} one defines
\begin{equation}\label{eq: 3-limit integrals}
 \int^{z_1}\int_{z_2}^{\gamma_2 z_2}\omega_{f_0}^\beta:=\frac{1}{e} I_{z_1,z_2}^\beta
(\partial^{-1}(e\cdot (1\otimes (\gamma_2-1))))\in \C/\Lambda_{f_0}^\beta.
\end{equation}
This is indeed a well-defined quantity in $\C/\Lambda_{f_0}^\beta$, because any two preimages of $e
(1\otimes (\gamma_1-1))$ by $\delta$ differ by an element of $Z_2(X,\Z)$. It is sometimes convenient to use expressions
such as $\int^x\int_y^z\omega_{f_0}^\beta$, but we warn the reader that they only make sense if $z=\gamma y$ for
some $\gamma\in \Gamma$ and $e\cdot (1\otimes \gamma)\in \im(\delta)$ for some $e$. It follows from the definitions that
the $3$-limit integrals of \eqref{eq: 3-limit integrals} enjoy the following properties:
\begin{eqnarray}
\int^{x}\int_y^{z} &\omega_{f_0}^\beta &=\int^{\gamma x}\int_{\gamma y}^{\gamma z}
\omega_{f_0}^\beta\ \   \text{ for all } \gamma\in \Gamma,\label{eq: property 1 of the integral}\\
% \item $$\int^{x}\int_y^{z} (\omega_f^\pm + \omega_g^\pm )=\int^{-1/(Nx)}\int_{-1/(Ny)}^{-1/(Nz)}
% (\omega_f^\pm + \omega_g^\pm ).$$
\int^{x}\int_y^{z} &\omega_{f_0}^\beta &=\int^{x}\int_y^{t}
\omega_{f_0}^\beta+\int^{x}\int_t^{z}
\omega_{f_0}^\beta,\label{eq: property 2 of the integral}\\
\int^{y}\int_t^{z} &\omega_{f_0}^\beta &-\int^{x}\int_t^{z}
\omega_{f_0}^\beta=\int_x^{y}\int_t^{z}
\omega_{f_0}^\beta.\label{eq: property 3 of the integral}
\end{eqnarray}

Now let $K/F_0$ be a quadratic ATR extension  and let $\varphi\colon R_c\hookrightarrow \cO$ be a normalized optimal
embedding of conductor $c$. Denote by $z_1$ the unique fixed point  of $K$  acting on $\mathcal H_1$ through $v_1$.
The stabilizer $\Gamma_{\varphi}$ of $z_1$ in $\Gamma$ is an abelian group or rank $1$ (cf. \cite[Proposition
1.4]{DL}). Call $\gamma_\varphi$ one of its generators. Let $z_2,z_2'\in \partial \mathcal H_2$ denote the two fixed
points of $K$ acting through $v_2$. Then we have that
\[
 \int_{\tilde{T}_\varphi}\omega_{f_0}^\beta=\int^{z_1}\int_{z_2}^{z_2'}\omega_{f_0}^\beta=\int^{z_1}\int_{
z_2}^{\gamma_{\varphi} z_2 } \omega_ { f_0 } ^\beta.
\]
Using properties \eqref{eq: property 1 of the integral}, \eqref{eq: property 2 of the integral} and \eqref{eq: property
3 of the integral}  it is easy to check that the last integral does not depend on $z_2$. Therefore,
we see that
\begin{equation}\label{eq: definitioin of the point using 3-limit integrals}
 \int_{\tilde{T}_\varphi}\omega_{f_0}^\beta=\int^{z_1}\int_{x}^{\gamma_{\varphi}x}\omega_{f_0}^\beta
\end{equation}
for any $x\in \mathcal H_2\cup \P^1(F_0)$.
If $N=1$ an algorithm for computing $3$-limit integrals as the one in \eqref{eq: definitioin of the point using 3-limit
integrals} is given in   \cite[\S 4]{DL}, by means of the  \emph{continued fractions trick}. To the best of our
knowledge, for arbitrary level $N$ at the moment no
generalization of this algorithm is known (cf. also \cite[Annexe A2]{Ga}).

\section{Almost totally complex points}\label{secATC}
This section  is devoted to the main construction of the article. It is an explicitly computable construction of points
on certain elliptic $F_0$-curves. By granting conjectures of \S \ref{section: Quadratic points on modular abelian
vareties}  over ATR extensions, these points are shown to be rational over ATC fields.
Recall that for a field extension  $F/F_0$, an elliptic curve $E/F$ is said to be an
\emph{elliptic $F_0$-curve} if it is isogenous over $F$ to all of its $\Gal(F/F_0)$-conjugates.

The construction of ATC points on $F_0$-curves is given in \ref{subsection: construction of the points}. In
\ref{ComparingGartner} we compare our ATC points  with
Gartner's Darmon points,
and conjecture a precise relation between them. Finally, in
\ref{subsection: numerical examples} we carry out an explicit calculation of such an ATC point for a particular elliptic
curve. At the same time of giving the details of how explicit computations can be handled, we numerically verify that
the obtained point satisfies the conjectures of \S \ref{section: Quadratic points on modular abelian vareties},
which provides certain evidence for their validity.

\subsection{Construction of ATC points}\label{subsection: construction of the points}
Let $F_0$ be a totally real number field of narrow class number $1$ and degree $r$. We denote by $v_1,\dots, v_r$
the embeddings of $F_0$ into $\C$, and we regard $F_0$ as a subfield of $\C$ via $v_1$. We will also regard all
extensions of $F_0$ as embedded in $\C$ via a fixed extension of $v_1$ to $\overline{F}_0$, which we denote by $v_1$ as
well. Let $F=F_0(\sqrt{N_0})$ be a totally real
quadratic extension and let $E/F$ be an elliptic $F_0$-curve without complex multiplication.

We denote by
$A=\Res_{F/F_0}E$  the
variety over $F_0$ obtained by restriction of scalars. If $E$ is not isogenous to the base change of an
elliptic curve defined over $F_0$, then $A/F_0$ is simple and $\Q\otimes\End_{F_0}(A)$ is isomorphic to a
quadratic field. From now on we restrict to the following setting.
\begin{ass}
$\Q\otimes\End_{F_0}(A)$ is a quadratic imaginary field.
\end{ass}
 We shall also make
the following assumption, which is a consequence of the generalized Shimura--Taniyama
Conjecture for abelian varieties of $\GL_2$-type.

\begin{ass}
There exists a normalized Hilbert modular form $f_0$  over $F_0$ of parallel weight $2$ such that $A$ is isogenous to
$A_{f_0}$ (where we recall that $A_{f_0}$ is the modular abelian variety attached to $f_0$ by means of the generalized
Eichler--Shimura construction, cf. Assumption \ref{ass1}).
\end{ass}
Therefore we can suppose that  $A=A_{f_0}$. Observe that, since $E$ is an $F_0$-curve, we have that $A\sim_F E^2$.
Denote by $N$ and $\psi$ the level and the nebentypus of $f_0$ respectively and, for an ideal $\mathfrak m$ of
$F_0$, denote by $a_\mathfrak{m}$ the Fourier coefficient of $f_0$ corresponding to $\mathfrak m$.

\begin{lemma}\label{lema: psi is quadratic}
 The character $\psi$ is quadratic and $F$ is the field corresponding by class field theory to the kernel of $\psi$.
\end{lemma}
\begin{proof}
  Denote by $F_\psi$ the field cut by the kernel of $\psi$. Let $G=\Gal(\Qb/F_0)$, $H=\Gal(\Qb/F)$ and
$H_\psi=\Gal(\Qb/F_\psi)$.  It is enough to show that $H=H_\psi$ (the fact that $\psi$ is quadratic follows from this
because $[F\colon F_0]=2$). Let $\ell$ be a
prime number that splits in $\Q_{f_0}$, say as $\ell=\lambda\lambda'$, and denote by $V_\ell=T_\ell(A)\otimes_{\Z_\ell}
\Q_\ell $ the $\ell$-adic Tate module of $A$. There is an isomorphism of $\Q_\ell[G]$-modules $V_\ell=V_\lambda\times
V_{\lambda'}$, where $V_\lambda=E_\lambda\otimes_{E\otimes \Q_\ell} V_\ell$ and
$V_{\lambda'}=E_{\lambda'}\otimes_{E\otimes \Q_\ell} V_\ell$.

Denote by $\rho_\lambda$
and $\rho_\lambda'$ the  representations of $G$ afforded by $V_\lambda$ and $V_{\lambda'}$ respectively, which are
irreducible because $E$ is not CM.
Since $A$ is
the variety attached to $f_0$ by the Eichler--Shimura construction, and relabeling $\lambda$ and
$\lambda'$ if necessary, we can suppose that:
\begin{equation}
 \Tr(\rho_\lambda(\Frob_{\mathfrak p}))=a_{\mathfrak p}\ \  \mbox{ and } \ \ \ \Tr(\rho_{\lambda'}(\Frob_{\mathfrak
p}))=\overline{a}_{\mathfrak p},\ \ \mbox{
for all primes ${\mathfrak p}\nmid N$},
\end{equation}
where the bar denotes complex conjugation. By \cite[Theorem 2.5]{Sh} the nebentypus $\psi$ is characterized by the fact
that
$a_{{\mathfrak p}}=\overline{a}_{\mathfrak p}\psi({\mathfrak p})$ for primes ${\mathfrak p}\nmid N$. Therefore
$V_\lambda$ and $V_{\lambda'}$ are isomorphic as
$\Q_\ell[H_\psi]$ representations, so that $\End_{\Q_\ell[H_\psi]}V_\ell\simeq \M_2(\Q_\ell)$. Moreover, $H_\psi$ is the
largest subgroup of $G$ for which this is true. On the other hand, we have that $\End_F(A)\otimes\Q\simeq
\End_F(E^2)\otimes\Q\simeq \M_2(\Q) $. By the case of Tate's Conjecture proven by Faltings
this implies that $\End_{\Q_\ell[H]}V_\ell \simeq \End_{F}^0(A)\otimes_\Q\Q_\ell\simeq \M_2(\Q_\ell)$, from which we
deduce that
necessarily
$H=H_\psi$.
\end{proof}
Observe that, as a consequence of the conductor-discriminant formula, $F$ has discriminant $N$ over $F_0$. For
simplicity we
assume from now on that $N$ is not divisible by any dyadic prime, and thus  squarefree.

Let $M=F(\sqrt\alpha)$ be  a quadratic ATC extension of $F$. Recall that ATC
stands for \emph{almost totally complex},
and it means in this case that $M$ has exactly two real places. We suppose that $M$ is real under the place $v_1$. We
aim to give an explicitly computable construction of
points in
$E(M)$, by making use of the conjectural constructions of Section \ref{section: Quadratic points on modular abelian
vareties}.

 Write $\Gal(F/F_0)=\{1,\tau\}$ and let $M'=F(\sqrt{\alpha^\tau})$.
Clearly $M$ is not Galois over $F_0$, and its   Galois closure $\cM$  is the composition of $M$ and
$M'$. It is easily seen that  $\Gal(\cM/F_0)\simeq D_{2\cdot 4}$, the dihedral group of order 8. The field
$K=F_0(\sqrt{\alpha\alpha^\tau})$ is
contained in $\cM$, and there exist fields $L$ and $L'$ such that the diagram of subfields of $\cM/F_0$ is given by:
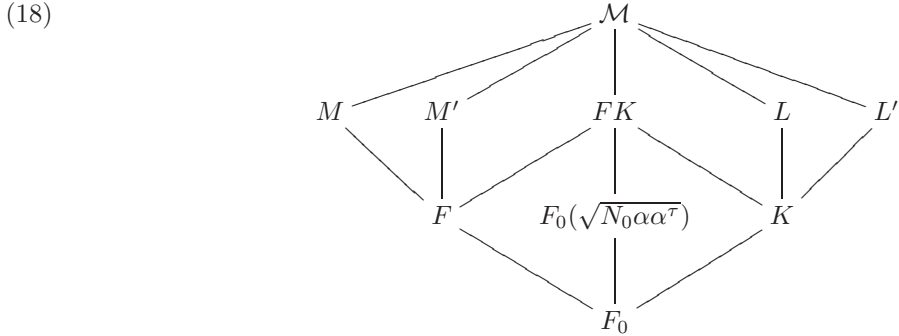
\begin{equation}\label{diagram of subfields}
\xymatrix{
         &           &  \cM \ar@{-}[d]\ar@{-}[dll]\ar@{-}[drr] \ar@{-}[dl]\ar@{-}[dr]     &         &           \\
M   & M'   &      FK      &    L    &    L'     \\
  &  F\ar@{-}[ul]\ar@{-}[u]\ar@{-}[ur] & F_0(\sqrt{N_0\alpha\alpha^\tau}) \ar@{-}[u] & K
\ar@{-}[ul]\ar@{-}[u]\ar@{-}[ur]\\
    & &F_0 \ar@{-}[ul]\ar@{-}[u]\ar@{-}[ur] & &
}
\end{equation}
Our construction relies on the fact that $K$ is ATR. Indeed, we have the following lemma.
\begin{lemma}\label{lemma: K ATR and L TC}
 The field $K$ is ATR and it is complex under $v_1$. The fields $L$ and $L'$ are totally imaginary.
\end{lemma}
\begin{proof}
 The first assertion  follows immediately from the definitions. The property about $L$ comes from the fact that it
can be identified with $K(\sqrt{\alpha}+\sqrt{\alpha^\tau})=K(\sqrt{\alpha+\alpha^\tau+2\sqrt{\alpha\alpha^\tau}})$,
and similarly for $L'$. Since $M$ is ATC, under a complex embedding of $L$ the image of either $\sqrt{\alpha}$ or
$\sqrt{\alpha^\tau}$ does not lie in $\R$.
\end{proof}

Since $K$ is an ATR extension which is complex under $v_1$ we are in the setting of
\S \ref{subsection: Darmon's algorithm for ATR points}.  Let $c\subset R_0$ be an integral ideal and let $R_c$ be
the order of conductor $c$ in $R_K$. Let $\cO$ be the
Eichler order of level $N$ in $M_2(R_0)$ consisting on matrices which are upper triangular modulo $N$, and let
$\varphi\colon R_c\hookrightarrow \cO$ be an optimal embedding. Observe that the
points $P_\varphi^\beta$ constructed in Section \ref{section: Quadratic
points on modular abelian vareties} are explicitly computable in this case, because Assumption \ref{GZass} holds true. Moreover, granting Conjecture \ref{ATR-gen}, they
belong to $A(H_\varphi^\beta)$. The key point  is that, as we shall see in Proposition
\ref{proposition: main
results about the fields}, for suitable choices of $c$ and $\beta$ the field $M$ is contained in $H_\varphi^\beta$.
Therefore,   points in $E(M)$ can be constructed by projecting $P^\beta_\varphi$ via the isogeny $A\sim_F
E^2$, and then taking trace over $M$.

Before stating and proving Proposition \ref{proposition: main results about the fields} we need some preliminary
results. Let $\chi_M,\chi_{M'}\colon \A_F^\times\ra \{\pm 1\}$ and $\chi_L,\chi_{L'}\colon \A_K^\times\ra \{\pm 1\}$
denote the
quadratic Hecke characters corresponding to the fields $M$, $M'$, $L$ and $L'$.  Similarly, let
$\varepsilon_F,\varepsilon_K\colon \A_{F_0}^\times \ra \{\pm 1\}$ be the ones corresponding to $F$ and $K$. Recall that
$\varepsilon_F=\psi$ by Lemma \ref{lema: psi is quadratic}.

\begin{lemma}\label{lemma: central character}
 \begin{enumerate}
\item $\chi_L\chi_{L'}=\psi\circ \mathrm{Nm}_{K/F_0}$.
  \item The central character of $\chi_L$ is $\psi$.
 \item We have that $\Ind_{F}^{F_0}\chi_M\simeq \Ind_{F}^{F_0}\chi_{M'} \simeq \Ind_{K}^{F_0}\chi_L\simeq
 \Ind_{F}^{F_0}\chi_{L'}$ are isomorphic as representations of $\Gal(\cM/F_0)$.
 \end{enumerate}
\begin{proof}
Assertion  (1) follows from the fact that $\chi_L\chi_{L'}$ is the quadratic character associated with the extension
$FK/K$, which is $\psi\circ \mathrm{Nm}_{K/F_0}$. If we let $\sigma$ denote the generator of $\Gal(K/F_0)$, then
we have that $\chi_L(x^\sigma)=\chi_{L'}({x})$. Then from (1) we see that $\chi_L$ restricted to
$\mathrm{Nm}_{K/F_0}\A_K^\times$ is equal to $\psi$. Then by class field theory the central character of $\chi_L$
is either $\psi$ or $\psi \varepsilon_K$. But it cannot be $\psi\varepsilon_K$: let $u=(-1,1,\cdots,1)\in
\A_{F_0,\infty}^\times$ (where the first position corresponds to the place $v_1$). Then $\psi\varepsilon_K(u)=-1$, but
$\chi_L(u)=1$ because $v_1$ extends to a complex embedding of $K$.
 Finally, (3) follows from the fact that the group
 $D_{2\cdot 4}$ has a unique $2$-dimensional irreducible representation.
\end{proof}

\end{lemma}

\begin{proposition}\label{proposition: the discriminant of L/K}
 Let $\mathfrak d_{L/K}$ denote the discriminant of $L/K$. Then $\mathfrak
d_{L/K}=c\cdot\mathfrak N$, where $c$ is an ideal of $F_0$ and
$\mathrm{Nm}_{K/F_0}\mathfrak N=N$.
\end{proposition}
\begin{proof}
 By the conductor-discriminant formula $\mathfrak d_{L/K}$ equals the conductor of $\chi_L$. Then the proposition is a
consequence of the fact that the central character of $\chi_L$ is $\psi$, which has conductor $N$. We give the precise
statements from which Proposition \ref{proposition: the discriminant of L/K} follows as Lemma \ref{lemma:
technical lemma on the fields} below.
\end{proof}

\begin{lemma}\label{lemma: technical lemma on the fields}
\begin{enumerate}[(1)]
 \item If $p\subset F_0$ is a prime  such that $p\mid N$, then either  $p$ splits or ramifies in  $K$. In both
cases, exactly  one of the primes above $p$  exactly divides the conductor of $\chi_L$.
\item Let $p\subset F_0$ be a prime  such that $p\nmid N$ and $\p^e$ divides exactly the conductor of $\chi_L$  for some
prime $\p\subset K$ above $p$. Then either $p$ is inert in $K$ or splits as $p\cdot R_K=\p\p'$ and $(\p')^e$
divides exactly the conductor of $\chi_L$.
\end{enumerate}
\end{lemma}
\begin{proof}
To prove  (1), let $p$ be a prime of $F_0$ dividing $N$. If $p$ splits as $\mathfrak
p\mathfrak p'$ in $K$ then by Lemma \ref{lemma: central character} the composition
\[
 R_{0,p}^\times \lra R_{K,\p}^\times\times  R_{K,\p'}^\times \stackrel{\chi_{L,\p}\cdot\chi_{L,\p'}}{\lra} \{\pm 1\}
\]
equals $\psi_p$. Since by assumption $p$ is not dyadic and $N$ is squarefree,   $\psi_p$ is the unique character
or order $2$
of $R_{0,p}^\times/(1+p)$. Since $R_{K,\p}^\times/(1+\mathfrak p)\simeq R_{K,\p'}^\times/(1+\mathfrak
p')\simeq R_{0,p}^\times/(1+p)$ we see that the character
\[
 \begin{matrix}
  R_{0,p}^\times/(1+p)\times R_{0,p}^\times/(1+p) & \stackrel{\chi_{L,\p}\cdot\chi_{L,\p'}}{\lra} & \{\pm 1\}\\
(x,x) & \longmapsto & \chi_{L,\p}(x)\cdot \chi_{L,\p'}(x)
 \end{matrix}
\]
has order $2$. This implies that exactly one of $\chi_{L,\p}$ or $\chi_{L,\p'}$ is trivial. Suppose that
$\chi_{L',\p}$ is trivial
and $\chi_{L,\p}$ has order $2$. Then $\p$ divides exactly the conductor of $\chi_L$ and
$\p'$ does not divide it.

Suppose now that $p\mid N$ is ramified in $K$ so that  $pR_K=\p^2$. Then by Lemma \ref{lemma: central character} the
composition
\[
 R_{0,p}^\times {\lra} R_{K,\p}^\times \stackrel{\chi_{L,\p}}{\lra} \{\pm 1\}
\]
equals $\psi_p$, which is a character of order $2$ factorizing through $R_{0,p}^\times/(1+p)$. This implies that
$\chi_{L,\p}$ necessarily factorizes through $R_{K,\p}^\times/(1+\p)$, because $R_{K,\p}^\times/(1+\p)\simeq
R_{0,p}^\times/(1+p)$. Therefore $\p$ divides exactly
the conductor $\mathfrak d_{L/K}$ of $\chi_L$.

Suppose now that $p\mid N$ is inert in $K$, so that $pR_K=\p$. Again by Lemma \ref{lemma: central character}  the
character $\psi_p$ equals
\begin{equation}\label{eq: composition of local characters}
 R_{0,p}^\times \longrightarrow R_{K,\p}^\times \stackrel{\chi_{L,\p}}{\lra}
\{\pm 1\},
\end{equation}
 the composition of the natural inclusion with $\chi_{L,\p}$. But the map in \eqref{eq: composition of local characters}
is trivial. Indeed, in this case $\F_p^\times=R_{0,p}^\times/(1+p)$ is strictly contained in
$\F_\p^\times=R_{K,\p}^\times/(1+\p)$. Then  $\chi_{L,\p}$ is the unique quadratic character of $\F_\p^\times$, and such
character is
always trivial on $\F_p^\times$. The fact that $\psi_p$ is trivial  contradicts the fact that $p\mid N$, so this
case does not occur.

To prove (2) we use again that the localization at $p$ of the composition
\begin{equation}\label{eq: composition of hecke characters}
 \A_{F_0}^\times \lra \A_K^\times  \stackrel{\chi_L}{\lra} \{\pm 1\}
\end{equation}
coincides with $\psi_p$, and therefore it is trivial  because in this case $p\nmid N$. But $\chi_{L,\p}$ has order $2$,
so that in
particular it is not trivial. Suppose that  $\chi_{L,\p}$ has conductor $\p^e$ for some $e\geq 1$. Observe that now,
since $\p$ can be dyadic, the exponent $e$ may be greater than $1$ (in fact, it is equal to $1$ except if $\p$ is
dyadic, in which case it may also be $2$ or $3$). In any case, the localization of
\eqref{eq: composition of hecke characters} at $\p$ is trivial only in one of the following situations:
\begin{enumerate}[i)]
 \item The inclusion $\A_{F_0}^\times \lra \A_K^\times$ localizes to a strict
inclusion $R_{0,p}^\times/(1+p^e)\hookrightarrow R_{K,\p}^\times/(1+\p^e)$.
\item The map in \eqref{eq: composition of hecke characters} localizes to
\[
  \begin{matrix}
 R_{0,p}^\times/(1+p^e)&\lra & R_{0,p}^\times/(1+p^e)\times R_{0,p}^\times/(1+p^e) &
\stackrel{\chi_{L,\p}\cdot\chi_{L,\p'}}{\lra} & \{\pm 1\}\\x&\longmapsto &
(x,x) & \longmapsto & \chi_{L,\p}(x)\cdot \chi_{L,\p'}(x)
 \end{matrix}
\]
and $\chi_{L,\p}=\chi_{L,\p'}$.
\end{enumerate}
In the first case $p$ is inert in $K$. In the second case $p$
splits as $p\cdot R_K=\p\p'$ and $(\p')^e$ divides exactly the conductor of $\chi_L$.

\end{proof}

\begin{proposition}\label{proposition: main results about the fields} Let  $\varphi\colon
R_c\hookrightarrow \cO$ be a normalized optimal embedding, with $c$ as in Proposition \ref{proposition: the
discriminant of L/K}. The field $L_\varphi^\beta$ contains  $L$ if and only
if $\beta_j(-1)=-1$ for $j=2,\dots,r$.
\end{proposition}
\begin{proof}
 Recall that $$U_c = {\hat R_0}^\times (1+c{\hat R_K}) \subset \hat{K}^\times$$
and that
\[
 U_\varphi^+=\{\beta \in U_c \text{ such that } (\beta)_{\mathfrak N}\in \ker(\psi)\subset (R_0/NR_0)^\times\},
\]
where $\psi$ is the nebentypus of $f_0$ and also the character corresponding to the quadratic extension $F/F_0$.
Here $(\beta)_{\mathfrak N}$ denotes the image of the local term of the id\`ele $\beta$ in the quotient
$R_{K,\mathfrak N}^\times/(1+{\mathfrak N}\cdot R_{K,{\mathfrak N}})\simeq (R_0/N R_0)^\times$. The field
$L_\varphi^\beta$ is defined by
\[
 \Gal(L_\varphi^\beta/K)\simeq \A_K^\times/K^\times U_c^+K_\infty^\beta,
\]
where $ K_\infty^\beta$ is as in \eqref{eq: def open compact subgroup}.
Now let $\chi_L\colon \A_K^\times \ra\{\pm 1\}$ be the quadratic character corresponding to $L$. Observe that by class
field theory $L\subset L_\varphi^\beta$ if and only if $U_\varphi^+ K_\infty^\beta\subset \ker \chi_L$.

Let $\chi_L=\prod_v \chi_{L,v}$ be the decomposition of $\chi_L$ as a product of local characters. By the
conductor--discriminant formula the conductor of $\chi_L$ is equal to   $\mathfrak d_{L/K}=c\cdot {\mathfrak N}$. This
means that
$\chi_{L,f}=\prod_{\p\nmid \infty}\chi_\p$
factorizes through a character $$\chi_{L,f}\colon R_{K,c\cdot {\mathfrak N}}^\times /(1+c\cdot {\mathfrak N} R_{K,c\cdot
{\mathfrak N}})\ra \{\pm
1\}.$$
First of all we check that $\chi_L(U_\varphi^+\cap R^\times_{K,c\cdot{\mathfrak N}})=1$.
Let $a=(a_\p)_\p$ be an element in $U_\varphi^+\cap R^\times_{K,c\cdot{\mathfrak N}}$. We write it as $a=a_c\cdot
a_{\mathfrak N}$, where
$a_c=\prod_{\p\mid c}a_\p$ and $a_{\mathfrak N}=\prod_{\p\mid {\mathfrak N}}a_\p$.

If $\p\mid c$ then $\chi_{L,\p}(a_\p)=1$ by
the very definition of $U_\varphi^+$. Namely, if $e=v_\p(c)$ then $\chi_{L,\p}$ has conductor $\p^e$ so it can be
regarded as
a character $$\chi_{L,\p}\colon R^\times_{K,\p}/(1+\p^e R_{K,\p})\ra\{\pm 1\}.$$ But $a_\p$ belongs to
$(1+\p^e R_{K,\p})$ by the definition of $U_\varphi^+$, so that $\chi_{L,p}(a_\p)=1$.  Since this is valid for any
$\p\mid c$
we see that $\chi_L(a_c)=1$.

 Since ${\mathfrak N}$ has norm $N$ and $N$ is
squarefree we have that $R_{K,{\mathfrak N}}^\times/(1+{\mathfrak N}\cdot R_{K,{\mathfrak N}})\simeq (R_0/N
R_0)^\times$. Therefore the image of
$a_{\mathfrak N}$ via the map $\A_K^\times\ra R_{K,{\mathfrak N}}^\times/(1+{\mathfrak N}\cdot R_{K,{\mathfrak N}})$ can
be regarded as the image of some
$b\in \A_{F_0}^\times$ via the map $\A_{F_0}^\times \ra \A_K^\times\ra R_{K,{\mathfrak N}}^\times/(1+{\mathfrak N}\cdot
R_{K,{\mathfrak N}})$. By
Lemma \ref{lemma: central character} we have that ${\chi_L}_{|\A_{F_0}}=\psi$. Therefore, by the definition of
$U_\varphi^+$
we  see that $\chi_L(a_{\mathfrak N})=\psi(a_{\mathfrak N})=1$.

Since we have seen that $U_\varphi^+\subseteq \ker \chi_L$, we have that $L\subseteq L_\varphi^\beta$ if and only if $
\chi_L(K_\infty^\beta)=1$. It is clear that for the character $\beta$ such that $\beta_j(-1)=-1$ for $j=2,\dots,r$ this
is true, because then any character of $\A_{K,\infty}^\times$ is trivial when restricted to $K_\infty^\beta$. Suppose
now that $\beta$ is such that $\beta_j(-1)=1$ for some $j$. Then the $j$-th component of $K_\infty^\beta$ is equal to
$\R^\times$, and $\chi_L$ is not trivial restricted to this component because, by Lemma \ref{lemma: K ATR and L TC}, the
field $L$ is
totally imaginary so the  real place $v_j$  extends to a complex place of $L$.
\end{proof}
% \begin{remark}\label{remark: L prime}
%  The same argument shows that if  $\varphi$ is an  optimal embedding normalized with respect to $\overline{\mathfrak
% N}$, then $L'$ is contained in $L_\varphi^\beta$ if and only if $\beta_j(-1)=-1$ for $j=2,\dots,r$.
% \end{remark}

Now we let $c$  be as in Proposition \ref{proposition: main results about the fields}, and we take
$\beta\colon \Sigma\ra \{\pm 1\}$ to be the character such that $\beta_j(-1)=-1$ for $j=2,\dots, r$. Moreover we let
$\varphi\colon R_c\hookrightarrow \cO$ be an optimal embedding normalized with respect to $\mathfrak N$, with
$\mathfrak N$  as in Proposition \ref{proposition: the discriminant of L/K}. From now on we
grant Conjecture \ref{ATR-gen} so that $P_\varphi^\beta\in A(L_\varphi^\beta)$. Thanks to Proposition
\ref{proposition: main
results
about the fields}  we can set
\begin{equation*}
 P_{A,L}=\Tr_{L_\varphi^\beta/L}(P_\varphi^\beta)\in A(L).
\end{equation*}
If we denote by $C_L=\mathrm{rec}^{-1}(\Gal(L_\varphi^\beta/L))$, then by the reciprocity law of Conjecture
\ref{ATR-gen} $P_{A,L}$ can be computed as
\[
 P_{A,L}=\sum_{a\in C_L}(P_{a\star\varphi}^\beta)\in A(L).
\]

Observe that in Diagram \eqref{diagram of subfields} complex conjugation takes $L$ to $L'$. Therefore the point
$$P_{A,M}:=P_{A,L}+\overline{P_{A,L}}$$ lies in $A(M)$. Finally, we define
$$P_M=\pi(P_{A,M})\in E(M),$$
 where $\pi\colon A\ra E$ is the natural projection, an algebraic map
defined over $F$.

\begin{theorem}\label{conjecture: the point in E is non-torsion} Assume Conjecture \ref{ATR-gen}  holds true for the ATR
extension $K/F$.  Suppose also that the sign of the functional equation of $L(E/F,s)$ is $+1$ and that of $L(E/M,s)$ is
$-1$. Then $P_M$ is non-torsion if and only if $L'(E/M,1)\neq 0$.
\end{theorem}

\begin{proof}
If $L'(E/M,s)\neq 0$ then $L(E/M,s)$ vanishes with order $1$ at $s=1$. Since $$L(E/M,s)=L(E/F,s)L(E/F,\chi_M,s)$$ we see
that $L(E/F,\chi_M,s)$ vanishes with order $1$
at $s=1$. By  Lemma \ref{lemma: central character} we have that $\Ind_F^{F_0}\chi_M\simeq \Ind_{K}^{F_0}\chi_L$. Then
\begin{eqnarray*}
 L(E/F,\chi_M,s)&=&L(f_0/F\otimes\chi_M,s)=L(f_0\otimes
\mathrm{Ind_F^{F_0}\chi_M,s})\\&=&L(f_0\otimes
\mathrm{Ind_K^{F_0}\chi_L,s})=L(f_0/K\otimes \chi_L,s)\\&=&L(f_0/K,\chi_L,s),
\end{eqnarray*}
and therefore $L(f_0/K,\chi_L,s)$ vanishes with order $1$ at $s=1$. If we denote by  $\chi\colon
\Gal(L_\varphi^\beta/K)\ra \C$  the induction of $\chi_L$, then  part (2) of \ref{ATR-gen} implies that the point
\[
P_\chi= \sum_{\sigma \in
\Gal(L^\beta_\varphi/K)} \chi(\sigma)^{-1} \sigma(P^\beta_\varphi) \in A(L_\varphi^\beta)
\]
is non-torsion.

In order to apply the reciprocity law, let us view for a moment the fields $K$, $L$ and $L_\varphi^\beta$ as subfields
of $\C$
via a place of $\Qb$ extending $v_j$, for a fixed $j\in\{2,\dots,r\}$. Since $K$ is real under $v_j$ and $L$ is complex,
we see that complex conjugation
induces an element in $s\in \Gal(L_\varphi^\beta/K)$ that restricts to a generator of $\Gal(L/K)$. But $s$
corresponds under the reciprocity map to the id\`ele
\begin{equation}\label{cc-idele}
\xi_j:=(\xi_\infty,\xi_f)=(1,{\dots},1,\stackrel{j)}{-1},1,\dots,1)\times (1,1,\dots)\in K_\infty^\times\times \hat{K}^\times,
\end{equation}
so by part (3) of \ref{ATR-gen} we have that  $s(P_\varphi^\beta)=\beta(\xi_\infty)P_\varphi^\beta=-P_\varphi^\beta$. Then
we have that
\begin{eqnarray*}
 P_\chi &=& \sum_{\sigma \in
\Gal(L^\beta_\varphi/L)}  \sigma(P^\beta_\varphi)+ \sum_{\sigma \in
\Gal(L^\beta_\varphi/L)}  \chi(\sigma s)\sigma s(P^\beta_\varphi)\\&=&\sum_{\sigma \in
\Gal(L^\beta_\varphi/L)}  \sigma(P^\beta_\varphi)+ \sum_{\sigma \in
\Gal(L^\beta_\varphi/L)}  \chi_L( s)\sigma (-P^\beta_\varphi)\\&=&
2\cdot
\mathrm{Tr}_{L_\varphi^\beta/L}(P_\varphi^\beta)=2\cdot P_{A,L},
\end{eqnarray*}
which implies that $P_{A,L}$ is non-torsion. Moreover, as $s(P_\varphi^\beta)=-P_\varphi^\beta$ we have that
$P_{A,L}\in A(L)^{\chi_L}$. Then $P_{A,M}=P_{A,L}+\overline{P_{A,L}}$ belongs to $A(M)$ and is non-torsion as well.
Since the projection $\pi\colon A\ra E$ is defined over $F$ and $A(M)\simeq E^2(M)$, we see that $P_M=\pi(P_A(M))$
belongs to $E(M)$ and it is of infinite order.
\end{proof}

 Let $W_N$ denote the Atkin-Lehner involution on $S_2(\Gamma_\psi(N))$ corresponding to the ideal $N$.
By abuse of notation we also denote by $W_N$
the  involution that it induces on $A$. Then the splitting of the variety $A$ over $F$ is accomplished by
the
action of $W_N$. More precisely we have that
\[
 A\sim_F (1+W_N)A\times (1-W_N)A.
\]
Let $\lambda_N$  be the pseudoeigenvalue of $f_0$ corresponding to $N$; that is, the complex number satisfying
that $W_N(f_0)=\lambda_N \cdot \overline{f}_0$. Observe that the modular form
$$\alpha_{f_0}^\beta:=\frac{1}{1+\lambda_N}(f_0+W_N(f_0))$$ is
normalized. In view of Conjecture \ref{Oda-conj} the lattice of $E$ can be computed as
\begin{equation}\label{eq: lattice for the sum of the form and its atkin-lehner}\Lambda_{E}=(\Omega_2^{-}\cdots
\Omega_r^{-})^{-1}\cdot  \langle\int_{
Z}\alpha_{f_0}^\beta \rangle,
\end{equation}
where $Z\in H_2(X_\psi(\C),\Z) $ runs over the cycles  such that $ \int_Z(
\omega^\beta_{f_0}-W_N(\omega^\beta_{f_0}))=0$.  From this we obtain the following
explicit analytic formula for the points $P_M$.
\begin{theorem}\label{theorem: explicitc analytic formula}
Let
\begin{equation}\label{eq: explicit formula for J_M}
 J_\cM=(\Omega_2^{-}\cdots \Omega_r^{-})^{-1}\cdot \left(\sum_{a\in C_L}
\int_{\tilde{T}_{a\star \varphi}}\alpha_{f_0}^\beta \right).
\end{equation}
 Then the  point $P_M$ can be computed as
\begin{equation}\label{eq: explicit expression for P_M}
 P_M=\eta \left(J_\cM+\overline{J}_\cM\right),
\end{equation}
where $\eta$ is the Weierstrass parametrization $\eta\colon \C/\Lambda_{E}\ra
E(\C)$ and the bar denotes complex conjugation.
\end{theorem}
\begin{proof}
 The Atkin--Lehner involution $W_N$ is defined over $F$  and $FL=\cM$, so $\eta(J_\cM)$ belongs to $E(\cM)$.
We recall that we are viewing $\cM$ as a subfield of $\C$ by means of $v_1$. Under this embedding $\cM$ is complex and
$M$ is real and therefore
\[
 P_M=\Tr_{\cM/M}(\eta(J_\cM))=\eta(J_\cM)+\overline{\eta(J_\cM)}.
\]
Since $E$ is defined over $F$ and $F\subseteq \R$ we have that Weierstrass map commutes with complex conjugation, and
 \eqref{eq: explicit expression for P_M} follows.
\end{proof}

\begin{remark}
Observe that $W_N(P_L)=\mathrm{Tr}_{{L'}_\varphi^\beta/L'}(W_N(P_\varphi^\beta))$ belongs to $A(L')$. Since complex
conjugation does not fix $K$,  by part {\it (3)} of
Conjecture \ref{ATR-gen} we see that
$$\overline{P_{A,L}}=W_N(\sigma(P_{A,L}))+ P_t$$
for some $\sigma\in \Gal(L/K)$ and some $P_t\in A(L')_{\mathrm{tors}}$. If $\sigma$ turns out to be trivial and $P_t$
belongs to $A(F_0)_{\mathrm{tors}}$, then the point $P_{A,L}+W_N(P_{A,L})$ is already defined over $M$. In this case
$\eta(J_\cM)$ lies in $ E(M)$ and $P_M$ coincides, up to torsion, with $2\cdot \eta(J_\cM)$.  As we will see,
this is the situation encountered in the example of \S\ref{subsection: numerical examples}.
\end{remark}

\begin{remark}
Observe that the integral appearing in the formula of Theorem \ref{theorem: explicitc analytic formula} is completely
explicit. Indeed, in the case where $f_0$ has trivial nebentypus, an  algorithm for determining  the chains
$\tilde{T}_\varphi$ is
worked out in \cite{DL}, based on the approach taken in \cite{dasgupta-thesis}. As we showed in \S \ref{subsection: Darmon's algorithm for ATR points},
Darmon--Logan's method adapts to provide an explicit description of $\tilde{T}_\varphi$ also in the current
setting, in which $f_0$ has quadratic nebentypus.
\end{remark}

\subsection{Comparison with Gartner's ATC points}\label{ComparingGartner}

Let us keep the notations of the previous section \ref{subsection: construction of the points}; in particular $E/F$ is an elliptic curve defined over the totally real field $F$ and $M/F$ is an ATC quadratic extension. The curve $E$ is modular: its isogeny class corresponds by the Eichler-Shimura construction to the Hilbert modular form $f$ that one obtains from $f_0$ by base-change to $F$, in such a way that
$$
L(E,s)=L(f,s)
$$
as in \eqref{Lf}. Write $N_E\subseteq R_F$ for the conductor of $E$, that is to say, the level of $f$. It is related to the level $N$ of $f_0$ by the formula
\begin{equation}\label{NE}
\mathrm{Norm}_{F/F_0}(N_E)\cdot \disc(F/F_0)^2=N^2.
\end{equation}

We place ourselves under the hypothesis of Theorem \ref{conjecture: the point in E is non-torsion}, so that we assume $N_E$ is square-free, the sign of the functional equation of $L(E/F,s)$ is $+1$ and that of $L(E/M,s)$ is $-1$. 

As discussed in \S \ref{zintroduction}, our point $P_M$ in $E(M)$ is expected to coexist with another point
$P_M^{\mathrm{Gar}}$ (\cite[\S 5.4]{Gartner-article}), provided Conjecture \ref{ATR-gen} for the abelian extensions of
$M$ holds true. This point can be manufactured by applying the machinery of \S \ref{sec:quadratic}, \ref{sec:OdaShioda},
\ref{sec:Darmonpoints}, setting $M/F$ to play the role of the extension $K/F_0$ of loc.\,cit. 

Let us sketch the details: let $B$ be the quaternion algebra over $F$ which ramifies precisely at all the archimedean places of $F$ but $v_1$, $v_2$ (over which $M$ is complex) and at the prime ideals $\wp\mid N_E$ which remain inert in $M$. That this is a set of even cardinality is guaranteed by the sign of the functional equation of $L(E/M,s)$. Let $\cO$ be an Eichler order in $B$ of square-free level, divisible exactly by those primes $\wp\mid N_E$ which split in $M$.

Let $R_M$ denote the ring of integers of $M$ and fix a normalized optimal embedding $\varphi_M\in \mathcal E(R_M,\cO)$.
In the notations of \S \ref{sec:quadratic} and \ref{sec:OdaShioda} we have $r=2$ and $\Sigma=\{\pm 1\}$.  Take $\beta$
to be the trivial character  and, granting Conjecture \ref{ATR-gen}, let
$P_\varphi\in E(L_\varphi^\beta)$ denote the Darmon point associated with this choice. Set
\begin{equation}\label{GartnerPoint}
P_M^{\mathrm{Gar}} = \Tr_{L_\varphi^\beta/M}(P_\varphi^\beta) \in E(M).
\end{equation}

It is expected that the N\' eron-Tate height of $P_M^\mathrm{Gar}$ should be related to $L'(E/M, 1)$ while the N\' eron-Tate height of $P_M$ constructed in this paper should be connected to $L'(E/F, \chi_M, 1)$. Hence from the basic quality $$L'(E/M, 1) = L(E/F, 1) L'(E/F, \chi_M ,1),$$ we propose the following conjecture about the relation between $P_M$ and $P_M^\mathrm{Gar}$. Let $$\Omega_{E/F}= \frac{\prod_{\tau: F\hookrightarrow \R}c_{E^\tau}}{\sqrt{\disc(F)}}$$ 
where $c_{E^\tau}$ is either the real period or twice the real period of $E^\tau=E\times_\tau \R$, depending on whether $E^\tau(\zR)$ is connected or not.

\begin{conjecture}
The point $P_M^\mathrm{Gar}$ is of infinite order if and only if $P_M$ is of infinite order and $L(E/F, 1) \ne 0$. Moreover,
$$ P_M^\mathrm{Gar}= 2^s \ell \cdot P_M,$$
where $s$ is an integer which depends on $M$ and $\ell \in \zQ^\times$ satisfies $\ell^2 = \frac{L(E/F, 1)}{\Omega_{E/F}}.$
\end{conjecture}

\subsection{A numerical example}\label{subsection: numerical examples}
In this section we give the details for the computation of an ATC point on a particular elliptic curve. We  used Sage
\cite{sage} for all the
numerical calculations. We begin by describing the elliptic curve and
the corresponding Hilbert modular form $f_0$, which we will take to be the base change of a modular form $f$ over
$\Q$.

\subsubsection{The curve and the modular form}\label{subsec:quer}
Let $f$ be the (unique up to Galois conjugation) classical newform over $\Q$ of level $40$ and nebentypus
$\varepsilon(\cdot )=\left( \frac{10
}{\cdot}\right)$. It corresponds to the third form of level
$40$ in the table 4.1 of the appendix to \cite{Q}. We see from this table that
the modular abelian variety $A_{f}$ has dimension $4$. Moreover,  it breaks as the
fourth power of an elliptic curve $E/F$, where $F=\Q(\sqrt{2},\sqrt{5})$. Jordi Quer computed an equation
for $E$ using the algorithms of \cite{pep-lario}; a global minimal model  of $E$ is given by:
\begin{equation}\label{eq: equacio}
 y^2 + b_1 xy + b_3 y = x^3 + b_2 x^2 + b_4 x + b_6,\end{equation}
where
\begin{eqnarray*}
 b_1&=& 1-9/2\sqrt{2} + 3\sqrt{5} -1/2\sqrt{10},\\
b_2&=&-15/2 + 13/2\sqrt{2}  -9/2\sqrt{5}  + 5/2\sqrt{10},\\
b_3&=&-11/2 -27/2\sqrt{2} + 17/2 \sqrt{5}+  3/2\sqrt{10},\\
b_4&=& 41/2  +   8\sqrt{2} -15/2\sqrt{5}    -8\sqrt{10},\\
b_6&=&          525/2 +    8\sqrt{2} -13/2\sqrt{5}   -84\sqrt{10} .
\end{eqnarray*}

Let $F_0=\Q(\sqrt{2})$ and let $v_1$ (resp. $v_2$) be the embedding
taking $\sqrt{2}$ to the positive (resp. negative) square root of $2$. Since $E$ is a $\Q$-curve, it is also an
$F_0$-curve. If we set $\alpha={\sqrt{10}+\sqrt{5}+\sqrt{2}}$ then $M=F(\sqrt{\alpha})$ is an ATC extension of $F$.
Since the conductor of $E/F$ is equal to $1$ the sign of the functional equation of $L(E/F,s)$ is $+1$, and the sign of
$L(E/M,s)$ is $-1$. The point
$P_{\mathrm{nt}}\in E(M)$  whose $x$ coordinate is given by
\[
 x=\frac{-3259+ 2126\sqrt{\alpha} -8957\sqrt{\alpha}^2 +5297\sqrt{\alpha}^3 -4989\sqrt{\alpha}^4 +1954\sqrt{\alpha}^5
-743\sqrt{\alpha}^6+ 39\sqrt{\alpha}^7}{72}
\]
is a generator or the Mordell-Weil group of $E(M)$. Conjecture \ref{conjecture: the point in E is
non-torsion} predicts that the point $P_M$ coincides, up to torsion, with a multiple of $P_{\mathrm{nt}}$.
We computed an approximation to $J_\cM\in \C/\Lambda_E$ with an accuracy of $30$ decimal digits using formula \eqref{eq:
explicit formula for J_M}. Let $J_{\mathrm{nt}}\in \C/\Lambda_E$ be a preimage of $P_{\mathrm{nt}}$ by Weierstrass's
uniformization map. Then the following relation
\begin{equation}\label{eq: relation of points}
7\cdot J_\cM-14\cdot J_{\mathrm{nt}} \in\Lambda_E,
\end{equation}
holds up to the computed numerical precision of $30$ digits.  The torsion group
$E(M)_{\mathrm{tors}}$ is isomorphic to $\Z/14\Z$. Observe that  this gives numerical evidence for the fact that
$\eta(J_{M})$ is already a non-torsion point in $E(M)$ in this case. We find a similar relation for
$P_M=\eta(J_\cM+\overline{J_\cM})$:
 \[
7\cdot (J_\cM+\overline{J_\cM})-28\cdot J_{\mathrm{nt}} \in\Lambda_E.
\]
In the rest of the section provide the details about the computation of $J_\cM$,
beginning with those related to compute the Hilbert modular form attached to $E$ over $F_0$.

Let $f_0$ be the base change of $f$ to $F_0$. Denote by $N$  the level of $f$, and
let $A=\Res_{F/F_0}E$, which is a $\GL_2$-variety over $F_0$. By Milne's formula \cite[Proposition 1]{Mi3} it has
conductor $\mathrm{cond}(A/F_0)=(25)$. By the Shimura--Taniyama conjecture for $\GL_2$-type varieties $A$ is isogenous
to $A_{f_0}$, which has conductor $N^2$. Then we see that  $N=(5)$ and that $f_0$ belongs to $S_2(\Gamma_\psi(N))$,
where $\psi$ is the restriction of $\varepsilon$ to $\Gal(\Qb/F_0)$. By identifying $\varepsilon$ with a character
$\A_\Q^\times\ra\{\pm 1\}$ by means of class field theory, $\psi$ can be identified  with the id\`ele
character $\varepsilon\circ \mathrm{Nm}_{F_0/\Q}\colon \A_{F_0}^\times\ra\{\pm 1\}$.

The Fourier coefficients of $f=\sum_{n\geq 1} c_n q^n$ can be explicitly computed in Sage. Let us see how to compute
the  coefficients of $f_0$ in terms of the $c_n$'s. The field $\Q_{f}=\Q(\{c_n\})$  turns out to be $\Q(\sqrt
2,\sqrt{-3})$. Let $\Gal(\Q_{f}/\Q)=\{1,\sigma,\tau,\sigma\tau\}$, where $\sigma$ denotes the automorphism that fixes
$\Q(\sqrt{-3})$ and $\tau$ the one that fixes $\Q(\sqrt{2})$.  The
inner twists of $f$ are given by
$$\chi_\sigma=\varepsilon_{\Q(\sqrt{5})},\ \chi_\tau=\varepsilon_{\Q(\sqrt{10})},\ \chi_{\sigma\tau}=
\chi_\sigma\chi_\tau=\varepsilon_{\Q(\sqrt{2})},$$
 where $\varepsilon_{\Q(\sqrt{a})}$ denotes the Dirichlet character
corresponding to  $\Q(\sqrt{a})/\Q$. Recall that  inner twists are defined by the
relations $f^\rho=\chi_\rho\otimes f$. This is
also equivalent to say that $c_p^\rho=\chi_\rho(p)c_p$ for all $p$ not dividing the level of $f$ (see
\cite{Ri} for more details).
\begin{lemma}\label{lemma 1}
 $L(f_0,s)=L(f,s) L(f^{\sigma\tau},s).$
\end{lemma}
\begin{proof}
Indeed $f_0$ is the base change of $f$ to $F_0=\Q(\sqrt 2)$. Then,
$$L(f,s)=L(f,s)L(f\otimes
\varepsilon_{\Q(\sqrt{2})},s)=L(f,s)L(f\otimes
\chi_{\sigma\tau},s)=L(f,s)L(f^{\sigma\tau},s).$$
\end{proof}
The $L$-series of $f_0$ is of the form
\begin{equation}\label{eq: L-series of f}
 L(f_0,s)=\prod_{\p\nmid N} (1-a_\p \Norm(\p)^{-s}+\psi(\p)\Norm(\p)^{1-2s})^{-1}\prod_{\p\mid N}(1-a_\p
\Norm(\p)^{-s})^{-1},
\end{equation}

for some coefficients $a_\p$, indexed by the primes in $F_0$.
\begin{lemma}\label{lemma: a_n}
  Let $\p$ be a prime in $F_0$, and let $p=\p\cap \Z$. Then
\begin{equation*}
a_\p=
\begin{cases}
c_p &\text{if $\varepsilon_{\Q(\sqrt{2})}(p)=1\text{ and }\ p\neq 5$,}
\\ c_p^2-2\,\varepsilon(p)\, p & \text{if $\varepsilon_{\Q(\sqrt{2})}(p)=-1 \text{ and }\ p\neq 5$,}\\
 c_p^2 & \text{if $p=5$,}\\
c_p+c_p^{\sigma\tau} &\text{if $p=2$}.
\end{cases}
\end{equation*}

\end{lemma}
\begin{proof} For a rational prime $p$ let
$L_p(f_0,s)=\prod_{\p\mid p}L_\p(f,s)$ denote the product of local factors for the primes $\p\mid p$. For $p\neq 2,5$
Lemma
\ref{lemma 1} gives that
\begin{equation}\label{eq: local factor}
 L_p(f_0,s)=(1-c_pp^{-s}+\varepsilon(p) p^{1-2s})^{-1}(1-c_p^{\sigma\tau} p^{-s}+\varepsilon(p)
p^{1-2s})^{-1}.
\end{equation}

If $\varepsilon_{\Q(\sqrt{2})}(p)=1$ then $p$ splits in $F_0$ so there are two primes  $\p_1,\p_2 $ dividing $p$, each
one having norm $p$. On the other hand, $c_p^{\sigma\tau}=c_p$ (because
$\chi_{\sigma\tau}=\varepsilon_{\Q(\sqrt{2})}(p)$), and
$\psi(\p_i)=\varepsilon(\Norm(\p_i))=\varepsilon(p)$. Comparing \eqref{eq: local factor} and \eqref{eq: L-series of
f} we see that $a_{\p_i}=c_p$.

If $\varepsilon_{\Q(\sqrt{2})}(p)=-1$ then there is only one prime $\p$ dividing $p$, and
$\psi(\p)=\varepsilon(\Norm(\p))=1$. On the other hand $c_p^{\sigma\tau}=-c_p$, so
\begin{eqnarray*}
L_p(f_0,s)&=&(1-c_pp^{-s}+\varepsilon(p) p^{1-2s})^{-1}(1+c_p p^{-s}+\varepsilon(p)
p^{1-2s})^{-1}\\
&=&1+2\varepsilon(p)p^{1-2s}-c_p^2 p^{-2s}+p^{2-4s}\\
&=&1+(2\,\varepsilon(p)\,p-c_p^2) {\Norm(\p)}^{-s}+{\Norm(\p)}^{1-2s},
\end{eqnarray*}
and we see that $a_p=c_p^2-2\,\varepsilon(p)\, p$.

If $p=5$, then $\chi_{\sigma\tau}(p)=-1$ so $c_p^{\sigma\tau}=-c_p$. Since $5$ divides the level of $f$ we have that
\[
 L_p(f_0,s)=(1-c_pp^{-s})^{-1}(1+c_pp^{-s})^{-1}=(1-c_p^2{\Norm(\p)}^{-s})^{-1},
\]
 so that $a_\p=c_p^2$.

Finally, if $p=2$ then $(p)=\p^2$ in $F_0$. But $\p$ does not divide the level of $f_0$ and $\psi(\p)=-1$ (because $\p$
 is inert in $F$), so $L_\p(f_0,s)=L_p(f_0,s)$ is of the form
\begin{equation}\label{eq: p=5 primera}
 L_\p(f_0,s)=(1-a_\p p^{-s}-p^{1-2s})^{-1}.
\end{equation}
On the other hand, $p$ divides the level of $f$, so that
\begin{equation}\label{eq: p=5 segona}
 L_\p(f_0,s)=(1-c_pp^{-s})(1-c_p^{\sigma\tau} p^{-s})=(1-(c_p+c_p^{\sigma\tau})p^{-s}+c_pc_p^{\sigma\tau} p^{-2s}).
\end{equation}
It turns out that $c_pc_p^{\sigma\tau}=-p$, so \eqref{eq: p=5 primera} and
\eqref{eq: p=5 segona} match and we see that $a_\p=c_p+c_p^{\sigma\tau}$.
\end{proof}

\subsubsection{Computation of the ATC point}
 Let
$e=\sqrt{2}-1$ be a fundamental unit of $F_0$.  Observe that $e_1=v_1(e)>0$ and $e_2=v_2(e)<0$. Let $\beta\colon\{\pm
1\}\ra \{\pm 1\}$ be the nontrivial character. The differential $\omega_{f_0}^\beta$ is then the one corresponding to
\[
 \omega_{f_0}^\beta=\frac{-4\pi^2}{\sqrt{8}}\left(f_0(z_1,z_1)dz_1dz_2-f_0(e_1z_1,e_2\overline{z}
_2)d(e_1z_1)d(e_2\overline { z } _2)\right).
\]
As for $W_N(\omega_{f_0}^\beta)$,  it is easy to compute because $W_N(f_0)=\lambda_N \overline{f}_0$, where
the pseudoeigenvalue  $\lambda_N$   is equal to $a_{(N)}/N=\frac{-1+2\sqrt{-6}}{5}$. Therefore
\[
 W_N(\omega_{f_0}^\beta )=\frac{(4-8\sqrt{-6})\pi^2}{5\sqrt{8}}\left(\overline{f}_0(z_1,
z_1)dz_1dz_2-\overline{f}_0(e_1z_1 , e_2\overline{z}
_2)d(e_1z_1)d(e_2\overline { z } _2)\right).
\]
 and we have completely determined $\alpha_{f_0}^\beta=\omega_{f_0}^\beta+W_N(\omega_{f_0}^\beta)$.

Recall that $M$ is not Galois over $F_0$, and that the diagram of subfields of its Galois closure $\mathcal M$ is
the one given in \eqref{diagram of subfields}. The
 ATR field $K$ is easily computed to be $K=F_0(\omega)$, where $\omega^2 + (\sqrt{2} + 1)\omega + 3\sqrt{2} + 4=0$. Here
we remark that $K$ is complex under the embeddings extending $v_1$, and it is real under the embeddings extending
$v_2$.  The discriminant of $L/K$ is an ideal $\mathfrak N$ which in this case satisfies that
$\mathrm{Nm}_{K/F_0}(\mathfrak{N})=N$. Therefore the ideal $c$ of Proposition \ref{proposition: main results about the
fields} is equal to $1$ for this example. Let $\varphi\colon R_K\hookrightarrow \cO$
be the
optimal embedding of the maximal order $R_K$ into the Eichler order of conductor $N$ of $M_2(F_0)$ given by
\[
 \varphi(\omega)=\left(
\begin{array}{cc}
-\sqrt{2}+2 & -2\\
5 & -3
\end{array}
\right).
\]
By Proposition \ref{proposition: main results about the fields} we see that $L$ is contained in $L_\varphi^\beta$. But
$L_\varphi^\beta$ is a quadratic extension of the narrow Hilbert class field of $K$. Since $K$ turns out to have narrow
class number $1$, we see that $L_\varphi^\beta$ is a quadratic extension of $K$, hence equals $L$. This means
that $H_\varphi^\beta=\cM$, so that according to Conjecture \ref{ATR-gen} the point $P_\varphi^\beta$ is  defined
over $\cM$.

The fixed point of $K^\times$ under $\varphi$ (with respect to $v_1$) is
\[
 z_1\simeq 0.358578643762691 + 0.520981147679366\cdot i
\]
The unit
\[
e_K=(-10\sqrt{2} + 14)w + 7\sqrt{2} - 11
\] satisfies that $\mathrm{Nm}_{K/F_0}(e_K)=1$ and generates the group of such units, so that
\[
 \gamma_\varphi=\varphi(e_K)=\left(
\begin{array}{cc}
 -27\sqrt{2} + 37 & 20\sqrt{2} - 28\\
 -50\sqrt{2} + 70 & 37\sqrt{2} - 53
\end{array}
\right)
\]
and $$\gamma_\varphi\cdot \infty=\frac{-27\sqrt{2} + 37}{-50\sqrt{2} + 70}=\frac{4\sqrt{2}+11}{10}.$$ To compute $J_\cM$
 we need to evaluate the $3$-limits integral
\begin{equation}\label{eq: ATR point}
J_\cM= \int^{z_1}\int_\infty^{\gamma_\varphi\cdot \infty}
\alpha_{f_0}^\beta= \int^{{z_1}}\int_\infty^{\frac{4{\sqrt{2}}+11}{10}} \alpha_{f_0}^\beta.
\end{equation}
The next step is to use  properties \eqref{eq: property 1 of the integral}, \eqref{eq:
property 2 of the integral},  and \eqref{eq: property 3 of the integral} to transform \eqref{eq: ATR point} into a sum
of usual $4$-limit integrals, because they can be numerically computed by  integrating (a truncation of) the Fourier
series of
$\alpha_{f_0}^\beta$. Observe that $\alpha_{f_0}^\beta$ is
invariant under $W_N=W_{(5)}$, so we have the following additional invariance property:
\begin{equation}
 \int^{x}\int_y^z\alpha_{f_0}^\beta=\int^{\frac{-1}{5x}}\int_{\frac{-1}{5y}}^{\frac{-1}{5z}}\alpha_{f_0}^\beta.
\end{equation}
We will also use the following matrices, both belonging
to $\Gamma_\psi(N)$:
\[
G=\left(
\begin{array}{cc}
4{\sqrt{2}} + 11 & -3{\sqrt{2}} + 5\\
     10 &-6{\sqrt{2}} + 9
\end{array}
\right),\ \ \
H=\left(
\begin{array}{cc}
-15{\sqrt{2}} + 21   &  -{\sqrt{2}} - 1\\
   -35{\sqrt{2}} + 50    &       1
\end{array}
\right).
\]
Since $\gamma_\varphi\cdot \infty=G\cdot \infty$ and $G\cdot 0 =\frac{-3{\sqrt{2}} + 5}{-6{\sqrt{2}} +
9}={\sqrt{2}}/3+1$, we have that
\begin{equation}\label{eq: first part}
\begin{split}
 \int^{{z_1}}\int_\infty^{\gamma_\varphi\cdot\infty}\alpha_{f_0}^\beta=&\int^{{z_1}}\int_\infty^{G\cdot\infty}\alpha_{
f_0 }^\beta=\int^{ {z_1}}\int_\infty^{
G\cdot 0}\alpha_{f_0}^\beta+\int^{{z_1}}\int_{G\cdot 0}^{G\cdot\infty}\alpha_{f_0}^\beta\\ =&\int^{{z_1}}\int_\infty^{
{\sqrt{2}}/3+1}\alpha_{f_0}^\beta+\int^{G^{-1}\cdot{z_1}}\int_{ 0}^{\infty}\alpha_{f_0}^\beta.
\end{split}
\end{equation}
Now, since $H\cdot\infty=\frac{-15{\sqrt{2}} + 21}{-35{\sqrt{2}} + 50}=\frac{-3{\sqrt{2}}}{10} $ and  $H\cdot
0=-{\sqrt{2}}-1$ we have that
\begin{equation}\label{eq: second part}
\begin{split}
\int^{{z_1}}
\int_\infty^{{\sqrt{2}}/3+1}&\alpha_{f_0}^\beta=\int^{{z_1}-1}\int_\infty^{{\sqrt{2}}/3}\alpha_{f_0}^\beta=\int^{\frac{
-1}{ 5({z_1}-1)}} \int_0^{\frac{-3{\sqrt{2}}}{10}}\alpha_{f_0}^\beta
\\
=& \int^{\frac{-1}{5({z_1}-1)}}\int_0^{H\cdot\infty}\alpha_{f_0}^\beta = \int^{\frac{-1}{5({z_1}-1)}}\int_0^{H\cdot
0}\alpha_{f_0}^\beta+\int^{\frac{-1}{5({z_1}-1)}}\int_{H\cdot 0}^{H\cdot \infty}\alpha_{f_0}^\beta\\=&
\int^{\frac{-1}{5({z_1}-1)}}\int_0^{-{\sqrt{2}}-1}\alpha_{f_0}^\beta+\int^{H^{-1}\cdot\frac{-1}{5({z_1}-1)}}\int_{0}^{
\infty} \alpha_{f_0}^\beta\\
=&\int^{\frac{-1}{5({z_1}-1)}}\int_0^{\infty}\alpha_{f_0}^\beta+\int^{\frac{-1}{5({z_1}-1)}}\int_\infty^{-{\sqrt{2}}-1
}\alpha_{ f_0}^\beta+\int^{ H^{-1}\cdot\frac{
-1 } { 5({z_1}-1)}}\int_{0}^{\infty}\alpha_{f_0}^\beta\\
=&\int^{\frac{-1}{5({z_1}-1)}}\int_0^{\infty}\alpha_{f_0}^\beta+\int^{\frac{-1}{5({z_1}-1)}+{\sqrt{2}}+1}\int_\infty^{
0}\alpha_ {f_0}^\beta+\int^{ H^{-1}\cdot\frac{
-1 } { 5({z_1}-1)}}\int_{0}^{\infty}\alpha_{f_0}^\beta\\
=&\int_{\frac{-1}{5({z_1}-1)}+{\sqrt{2}}+1}^{\frac{-1}{5({z_1}-1)}}\int_0^{\infty}\alpha_{f_0}^\beta+\int^{H^{-1}
\cdot\frac{
-1 } { 5({z_1}-1)}}\int_{0}^{\infty}\alpha_{f_0}^\beta
\end{split}
\end{equation}
Now, putting together \eqref{eq: first part} and \eqref{eq: second part} we have that
\begin{equation}
 \begin{split}
   \int^{{z_1}}\int_\infty^{\gamma_\varphi\cdot\infty}\alpha_{f_0}^\beta=&\int^{G^{-1}\cdot{z_1}}\int_{
0}^{\infty}\alpha_{f_0}^\beta+\int_{\frac{-1}{5({z_1}-1)}+{\sqrt{2}}+1}^{\frac{-1}{5({z_1}-1)}}\int_0^{\infty}\alpha_{
f_0} ^\beta+\int^{H^{-1} \cdot\frac{
-1 } { 5({z_1}-1)}}\int_{0}^{\infty}\alpha_{f_0}^\beta\\
=&\int^{\frac{-1}{5G^{-1}\cdot{z_1}}}\int_{\infty}^{0}\alpha_{f_0}^\beta+\int_{\frac{-1}{5({z_1}-1)}+{\sqrt{2}}+1}^{
\frac{-1}{ 5({z_1}-1)}}
\int_0^ { \infty } \alpha_{f_0}^\beta+\int^{H^{-1}\cdot\frac{-1} { 5({z_1}-1)}}\int_{0}^{\infty}\alpha_{f_0}^\beta\\
=&\int_{\frac{-1}{5({z_1}-1)}+{\sqrt{2}}+1}^{\frac{-1}{5({z_1}-1)}}
\int_0^ { \infty }\alpha_{f_0}^\beta +\int_{\frac{-1}{5G^{-1}\cdot{z_1}}}^{H^{-1}\cdot\frac{-1} {
5({z_1}-1)}}\int_{0}^{\infty}\alpha_{f_0}^\beta
 \end{split}
\end{equation}
Now both of these integrals can be easily computed, because for $x,y\in \mathcal H$ one has that
\[
 \int_x^y\int_0^\infty=\int_x^y\int_0^{i/\sqrt{5}}+\int_x^y\int_{i/\sqrt{5}}^\infty=\int_{\frac{-1}{5x}}^{\frac{-1}{5y}}
\int_\infty^{ i/\sqrt { 5 } } +\int_x^y\int_{i/\sqrt{5}}^\infty,
\]
which are integrals with all of their limits lying in $\mathcal H$ and they can be computed by integrating term by
term the Fourier expansion.

Let $\Lambda_1$ and $\Lambda_2$ be the period lattices of $E$ with respect to $v_1$ and
$v_2$, and denote by $\Omega_1^+ , \Omega_2^+ $
the real periods  and $\Omega_1^-$, $\Omega_2^-$
the imaginary periods. Using the above limits we integrated the truncation of the Fourier expansion of
$\alpha_{f_0}^\beta$ up to ideals of norm
$160000$ obtaining
\begin{eqnarray*}
J_\cM= {\left(\Omega_2^-\right)^{-1}}\int^{z_1}\int_\infty^{\gamma_\varphi\infty} \alpha_{f_0}^\beta
&\simeq&
6.1210069519472105302223690235\\ &+& i\cdot 5.4381903029486320686211994460.
\end{eqnarray*}
Recall that  $J_{\mathrm{nt}}$ stands for the logarithm of $P_{\mathrm{nt}}$ in $\C/\Lambda_E$. The actual value is
$$J_{\text{nt}}\simeq 3.3835055058970249460140888086 + i\cdot 2.7190951514743160343105997232.$$

We have that
\begin{eqnarray*}
  7\cdot J_\cM -14\cdot J_{\text{nt}} +\Omega_1^+\simeq
  3.742356\cdot 10^{-27}
- i\cdot  3.23117\cdot 10 ^{-27},
\end{eqnarray*}
which is the numerical evidence for the fact that relation \eqref{eq: relation of points} holds and that, up to torsion,
$\eta(J_\cM)$ equals $2P_{\text{nt}}$.

\end{document}